\documentclass[a4paper,fleqn,leqno,11pt]{article}

\usepackage{color}
\newcommand{\bp}{\begin{proof}}
\newcommand{\ep}{\end{proof}}

\newcommand{\proofendsign}{$\Box$} 
\newenvironment{proof}{{\noindent \bf Proof }}{{\hspace*{\fill}\proofendsign\par\bigskip}}
\usepackage{verbatim}
\usepackage{latexsym,amssymb}
\usepackage[centertags]{amsmath}

\newcommand{\qed}{\hfill\Box}
\newcommand{\bi}{\bigskip}
\newcommand{\sm}{\smallskip}

\newcommand{\wh}{\widehat}
\newcommand{\wt}{\widetilde}
\newcommand{\ee}{\end{equation}}
\newcommand{\eea}{\end{eqnarray}}
\newcommand{\bean}{\begin{eqnarray*}}
\newcommand{\eean}{\end{eqnarray*}}
\newcommand{\suml}{\sum\limits}

\newif\ifpcte
\newcommand{\noi}{\noindent}

\newcommand{\ve}{\varepsilon}

  \newtheorem{theorem}{Theorem}
  \newtheorem{definition}{Definition}[section]
  
  \newtheorem{remark}[definition]{Remark}

  \newtheorem{cond}[definition]{Condition}
  \newtheorem{proposition}[definition]{Proposition}
  \newtheorem{lemma}[definition]{Lemma}
  \newtheorem{cor}[definition]{Corollary}

  \newcommand{\beCond}[2]{\Rand{\vspace{0,6cm}\tt #1}\begin{cond}[#2]
     \label{#1}}  \newtheorem{asss}{Set of Assumptions}

\newtheorem{xx}{\bf xxx}
\newcommand{\xxx}[1]{\begin{xx}{\bf #1}\end{xx}}

\newcommand{\CC}{\mathcal{C}}
\newcommand{\CD}{\mathcal{D}}
\newcommand{\E}{\mathbb{E}}
\newcommand{\R}{\mathbb{R}}
\newcommand{\Z}{\mathbb{Z}}
\newcommand{\N}{\mathbb{N}}
\newcommand{\CA}{\mathcal{A}}
\newcommand{\CB}{\mathcal{B}}
\newcommand{\CE}{\mathcal{E}}
\newcommand{\CF}{\mathcal{F}}

\newcommand{\CL}{\mathcal{L}}

\newcommand{\be}[1]{\begin{equation}\label{#1}}
\newcommand{\bew}[1]{\begin{equation*}\label{#1}}
\newcommand{\bea}[1]{\begin{eqnarray}\label{#1}}

\newcommand{\beL}[2]{\begin{lemma}[#2]\label{#1}}
\newcommand{\beD}[2]{\begin{definition}[#2]\label{#1}}
\newcommand{\beT}[2]{\begin{theorem}[#2]\label{#1}}
\newcommand{\beP}[2]{\begin{proposition}[#2]\label{#1}}
\newcommand{\beC}[2]{\begin{cor}[#2]\label{#1}}

\renewcommand{\Pr}{\mathrm{P}}
\newcommand{\Ex}{\mathrm{E}}
\newcommand{\1}{\mathbf{1}}
\renewcommand{\d}{\mathrm{d}}

\newcommand{\La}{\Longrightarrow}

\newcommand{\Ltoo}{{_{\displaystyle\longrightarrow\atop L\to\infty}}}
\newcommand{\Lto}{{_{\displaystyle\Longrightarrow\atop L\to\infty}}}

\newcommand{\tto}{{_{\displaystyle\longrightarrow\atop t\to\infty}}}

\newcommand{\tio}{{_{\displaystyle\longrightarrow\atop i\to\infty}}}

\newcommand{\Tvenull}{{_{\displaystyle\Longrightarrow\atop\ve\to 0}}}

\numberwithin{equation}{section}

\newcommand{\qad}{$\hfill \square$}

\begin{document}

\title{Multi-type spatial branching models for local self-regulation I: Construction and an exponential duality}

\thispagestyle{empty}

\author{Andreas~Greven, Anja~Sturm, Anita~Winter, Iljana~Z\"ahle}

\date{\today}

\maketitle

%
%
%

\begin{abstract}
We consider a spatial multi-type branching model in which individuals
migrate in geographic space according to random walks and reproduce according to a
state-dependent branching mechanism which
can be sub-, super- or critical depending on the local intensity of
individuals of the different types.
The model is a Lotka-Volterra type model with a spatial component and
is related to two models studied in \cite{BlathEtheridgeMeredith2007}
as well as to earlier work in \cite{Etheridge2004} and in
\cite{NeuhauserPacala1999}.
Our main focus is on the diffusion limit of small
mass, locally many individuals and rapid reproduction.
This system differs from spatial critical branching systems since it is
not density preserving and the densities for large times do not depend on the
initial distribution but mainly on the carrying capacities.

We prove existence of the infinite particle model and the system
of interacting diffusions as solutions of martingale problems or systems of stochastic equations.
In the exchangeable case in which
the parameters are not type dependent we show uniqueness of the solutions.
For that purpose we establish a new exponential duality.

\end{abstract}

\vspace{0.5cm}

\noi

\vspace{0.5cm}
\noi

{\tiny
\tableofcontents
}


\section{Introduction}
\setcounter{equation}{0}
In this paper we consider interacting stochastic processes
indexed by $\Z^d$  or some countable Abelian group which are obtained as the high density
fast branching limit
of particle models in which particles are assigned a type
and  the number of particles of each type at a site is changing
due to branching and migration. While particles migrate independently
of each other, we will choose the branching mechanism
{\em state-dependent} with local self-regulation for the offspring distribution
reflecting competition among subpopulations such that the model obtained
may exhibit even in low dimensions
\begin{itemize}
\item stable populations, and
\item coexistence of different types.
\end{itemize}

For a few decades the spatial distribution of biological populations with
a geographic structure has been modeled by
the spatial {\it Dawson-Watanabe superprocess} or super random walk
in which infinitesimally small particles are supposed to migrate through
space and branch independently. However, it is well-known that
in the subcritical case we have local extinction, in the supercritical case
local explosion. In the critical case in
dimensions $d\le 2$ the interaction due to migration is not strong
enough to counteract local extinction due to branching.  Worse, at rare sites
where the system is not extinct, populations
grow without any bound, building up very high peaks.
This also holds for multi-type versions of these models
for which in addition in the critical case in $d=1,2$
{the high peaks are build up by mono-type populations.}

A way out could be to condition the population to stay constant in the
colonies. This results in a system of spatially {\it interacting
Fleming-Viot processes}. This is a process where  particles are assigned
a type from a possibly continuous type space and the local frequency of types changes due to a
resampling mechanism. Unfortunately, we still have to deal with two
competing mechanisms, here migration and resampling, which once more
yield a dichotomy to the effect that in dimensions $d\le 2$ the model
tends locally to mono-type configurations (see e.g.\ \cite{CoxGreven1994a})
and only in high dimension $d \geq 3$ allows for local coexistence of
different types.

The extinction and formation of mono-type clusters turn out to be a drawback since most
biological populations happen to live in two (or even one)
dimensions, while at the same time locally stable populations with local coexistence of
different types are observed.

To model the latter features we have to give up the martingale property of a single
colony arising from critical branching or resampling. The reason is that
we have to take into account that
{\em resources are limited}. Therefore the key task is to
extend the well-studied ``neutral'' models by new mechanisms inducing a
{\em self-regulation} of the population size and population type
decomposition under the constraints of bounded resources.

{In physics such a model is referred to as a model for
(local) self-organized criticality. So one can hope for such a model that
it does not exhibit the dichotomy of either locally unbounded growth or
local extinction and furthermore that coexistence of different types is possible.}

There have been previous attempts to introduce a self-regulative mechanism
on the level of one-type branching processes without a spatial structure
(compare \cite{JagersKlebaner2000}). Here, we will use a set-up
which includes a geographic structure of the population.
In this framework we will observe, depending on the set of parameters chosen,
various behaviors which are observable in nature,
namely local extinction versus non-extinction, or locally the coexistence of
different types versus the formation of mono-type clusters depending on the
parameters of the model.

A starting point for a spatial mono-type model with local self-regulation in one dimension
is the following {\it logistic stochastic partial differential equation} which arises in
the large local population limit combined with a spatial rescaling of $\Z$ to $\R$:
\be{mutri}
   \frac{\partial u}{\partial t}
 =
   \frac{1}{6}\frac{\partial^2 u}{\partial x^2}+(K-\theta u)u+\sqrt{u}\dot{W},
\ee
where $\dot{W}$ is white-noise. The function
$u(t,x)$ describes the density of a population in
colony $x\in\R$ at time $t\ge 0$.
In \cite{MuellerTribe1994} the existence of a critical
capacity $K_c$ is shown such that for $K<K_c$, $\tau^{\mathrm{ext}}:=\inf\{t\ge
  0:\,u(t,0)= 0\}<\infty$, while for $K>K_c$,
  $\tau^{\mathrm{ext}}=\infty$.
This continuous site model can not be extended to higher dimensions.
However, in \cite{BolkerPacala1997} and \cite{BolkerPacala1999}
a similar model is introduced for a discrete geographic space. Bolker and Pacala
propose by simulations that
equilibria exist even in low dimensions for certain values of the
parameters. The latter is proved in \cite{Etheridge2004}.

On the other hand predator-prey models like the classical
Lotka-Volterra model have been shown to be able to explain the
coexistence of species.
In \cite{NeuhauserPacala1999} the Lotka-Volterra model is studied with a very specific
spatial structure having properties in common with a model known in mathematics as the
voter model. In particular for certain parameters
coexistence of both types even in low dimensions, more precisely $d=2,$ are
established.

This work \cite{NeuhauserPacala1999} by Neuhauser and Pacala is also the starting point of two
more recent papers \cite{HW07} and \cite{CoxPerkins2005}. In
\cite{CoxPerkins2005},
it is shown that suitable rescaling in space and time of the density of one
population leads to a limit that is described by super-Brownian motion with a
drift. The form of the drift is related to coexistence and survival of a rare
type in the original Lotka-Volterra model in \cite{NeuhauserPacala1999}.


Whereas in these models exactly one individual inhabits a
site in the geographic state space the present paper will
continue the theory with a multi-type spatial branching system
in a geographic space consisting of
different colonies. This means that in our model the possible number of
individuals  in each colony is not restricted, even
though the colony has in some sense a {\it carrying capacity}
due to limited availability of resources at this location which are
necessary for the fecundity of the different types in a varying degree.

A technical point which makes this model harder to investigate than other
spatial multi-type branching models treated so far is that we face here
non-linearities in the drift terms. In particular,
the system is not density preserving. Moreover, in the long run
the density will be only dependent on the carrying capacities but not
very much on the initial state which is the key parameter in the
neutral models.
\smallskip

{{}The main goal of the present first part of the paper is to establish
the existence of the
particle and the diffusion models on infinite geographic space, and to give an
analytical characterisation as solutions of martingale problems and systems of stochastic equations.
In the particular case where the parameters do not depend on the particles' type, we establish a new exponential duality
which allows to prove uniqueness.

}


\section{{{}Construction and
characterization of the models}}
\label{s.multilogistic}
The model we shall study describes the masses of $M\in\N$ different types of individuals in a population distributed over
colonies in a geographic space $G$ with generally countably many-- thus possibly infinitely many -- components.

The diffusion systems studied here arise as the limit of suitable particle models. Due to the fact that we also consider a (countably)  infinite
geographic space the concepts of the state spaces as well as of solutions to martingale problems and respective SDEs
require some modifications compared to the case of a single or finitely many
components.  Therefore we begin in Subsection~\ref{ss.prelinfsp}
with an introduction and discussion of these concepts.  We then introduce the particle system
in Subsection~\ref{ss.underpartsys}, the diffusion limit in Subsection~\ref{ss.iabranchdiff} and
results on the approximation of the diffusion limit via particle systems in Subsection~\ref{ss.difflimit}.
Finally, in Subsection~\ref{ss.expdual} we present an exponential duality relation for a special exchangeable case of our model,
which allows us in particular to make stronger statements about convergence to and uniqueness of the limit diffusion.

\subsection{Preliminaries on systems in infinite geographic space}
\label{ss.prelinfsp}
In this subsection we discuss  the adaptions needed to define a Markov process with
countably infinitely many interacting components.  We assume that the location of colonies in geographic space is given by some {{}countable}
Abelian group $G$. 
Thus the processes considered here will have state spaces which are subsets of
$(\R_+^M)^G.$ 
Two problems arise here: (1) the components live in the unbounded
set $\R_+$ and (2) the geographic space is infinite. This combination makes it more difficult to establish the
existence and uniqueness of the stochastic processes in general and  in particular given the additional
complication in our model with a nonlinear interaction.

For every system with state space contained in
$(\R_+^M)^G$ with $|G|=\infty$, the problem arises to set up the
state space in such a way that the dynamics can be well-defined
and no influence from infinity occurs at specific sites
rendering the process unspecified.

To keep the process locally finite
we choose as the state space the Liggett-Spitzer space (first
introduced in \cite{LiggettSpitzer1981}).

Let {for a countable Abelian group $G$}, $a(\cdot,\cdot)$ be {{}a random walk kernel from $G$ to $G$, i.e.,}
\be{sg1}
   a(\eta,\xi)=a(0,\xi-\eta)\quad\mbox{ and }\sum_{\xi\in G}a(0,\xi)=1,
\ee
which we use later to model migration on $G$.

Next we choose a weight function $\rho$ as follows
\be{Z035}
    \rho(\xi):=\sum_{\eta\in G}\sum_{n=0}^\infty (R/2)^{-n}\hat{a}^{(n)}(\eta,\xi)\beta(\eta)
\ee
with $R>2$, $\beta(\eta)>0$ for all $\eta\in G$, $\sum_{\eta\in G}\beta(\eta)<\infty$ and
\begin{equation}
\label{e:hata}
   \hat{a}(\xi,\eta):=\frac12(a(\xi,\eta)+a(\eta,\xi)).
\end{equation}
Note that $\rho$ {{}is positive and summable and for $\eta\in G$,
\be{Gr3}
\begin{aligned}
    & \sum_{\xi\in G}\hat{a}\big(\xi,\eta\big)\rho(\xi) \le \tfrac{R}{2}\rho(\eta).
\end{aligned}
\ee}

As state space we consider
\be{initial}
   {{}{\mathcal E}^G}
 :=
   \big\{x\in (\R_+^M)^G:\,\sum_{\xi \in G}\rho(\xi)
   \bar{x}_\xi<\infty\big\},
\ee
{{}where
\be{Z001}
    \bar{x}_\xi:=\sum_{m=1}^M x_\xi^m,
\ee
and write $\|x\|$ for the {\em $\rho$-weighted $l^1$-norm}, i.e., for all $x \in (\R_+^M)^G,$
\be{e:rhoweight}
   \|x\|:= \sum_{\xi\in G}\rho(\xi)\bar x_\xi.
\ee
}
We equip the state space $\mathcal{E}$ with
the {\em product topology} of $(\R_+^M)^G$.

We do not choose the norm topology
since we cannot expect to find a solution with regular paths 
in the state space $\CE^G$ equipped with the norm topology.

Finally, as a state space for the approximating particle systems we consider  the subset
\begin{equation}
\label{LLS}
   {{}{\mathcal E}^{\mathrm{par},G}:=  {\mathcal E}^G} \cap (\N_0^M)^G.
\ee

In the following and throughout the paper we will denote by $\CB(E),$ respectively $B(E)$ the set of all measurable, respectively measurable and bounded real valued functions on  a topological space $E$. {{}We further denote by $C(E,F)$ and $C_b(E,F)$ the space of continuous and continuous bounded functions from a space $E$ to another space $F$. If $F:=\R$, we simply write $C(E)$ or $C_b(E)$. }
In the case of $E=\R_+$ we use $D(\R_+,F)$ for the Skorohod space of c\`adl\`ag functions with values in $F$.

In order to make precise what we mean by the solution to a martingale problem we formulate:
\begin{definition}[Martingale problem]
Let a state space $\CE$, a set $\CF\subset C_b(\CE,\R)$ and a linear operator $\Omega_X$ 
with domain including  $\CF$ be given. Furthermore, let $\nu$ be a distribution on $\CE.$
Then, solutions to the $(\Omega_X, \CF, \nu)$ martingale problem are processes
$X$ with paths in $D(\R_+,\CE)$ such that
\be{ag2}
\big(f (X(t)) - \int^t_0 (\Omega_X f)(X(s))\mathrm{d}s\big)_{t \geq 0} 
\ee
is a martingale for every $f \in \CF$ and $ \CL [X_0] = \nu$. {{\em Uniqueness} holds for the martingale problem if there is at most
 one $P\in D(\R_+,\CE)$ such that under $P$, (\ref{ag2}) is a martingale for all $f\in\CF.$ The martingale problem is {\em well-posed} if there exists exactly one such $P.$}
\label{D.mp}
\end{definition}

As we shall see later, uniqueness
for the martingale problem  with infinitely many components
can be verified a priori only in very particular cases,
for example, via duality relations in the exchangeable case.
In order to address the uniqueness problem we will therefore, in general cases only consider
solutions which allow an approximation by spatially finite
systems.

We shall use approximations by populations that live in   finite geographic spaces
\be{ag1}
   (G_L)_{L \in \N}, \quad G_L \uparrow G, \quad |G_L| < \infty.
\ee
For example $G_L =[-L,L]^d \cap \Z^d$
for $G=\Z^d$.

\begin{definition}[Approximation property]
A solution $X$ of the $(\Omega^G_X, \CF, \nu)$-martingale problem has the approximation property (with respect to $\{G_L,\,L\in\mathbb{N}\}$)
if
there exists $\{X^L,\,L\in\mathbb{N}\}$ of ${\mathcal E}^G$-valued strong Markov processes with
\be{ag3}
   \CL \big[(X^L(t))_{t \geq 0}\big] \Lto \CL \big[(X(t))_{t \geq 0}\big],
\ee
and for each $L\in\mathbb{N}$, $X^L$ solves an
 $(\Omega^L_X, \CF, \nu)$-martingale problem for an operator $\Omega^L$ such that $\Omega^Lf(x)=0$ whenever $f\in{\mathcal F}$ and the restriction of $f$ to $G_L$ is a constant function.
\label{D.approx}
\end{definition}

\subsection{The underlying particle system}
\label{ss.underpartsys}

In this subsection we  introduce the approximating particle systems. These systems will also give an intuitive meaning to
the parameters that are used in the description of self-regulating population models.

We consider particles (individuals) that are assigned a {\em location} in the geographic space $G$ and
a {\em type} $m\in\{1,\dots,M\}$. If not stated otherwise
we consider $M\ge 2$. These particles are \emph{migrating} in the space $G$ and they
are also \emph{branching} (reproducing) in an environment of limited resources.
\sm

The branching is state-dependent meaning that
due to bounded resources the mean number of offspring varies
as a function of the current state, although we will assume for simplicity that
the branching is binary. In order to describe the state dependence
we need parameters that quantify the carrying capacity
for type $m\in\{1,\dots,M\}$ and the influence of type $n$ on
type $m$.

For each $m\in\{1,\dots,M\}$ the carrying capacity, $K^m \in (0,\infty)$, of
a colony for the $m^{\mathrm{th}}$ type arise as follows:
Assume there are $J$ different resources of respective size
\be{am1}
   R^j,\quad j = 1,\dots,J,
\ee
which the individuals have to share. Abundance  and shortage of resources cause additional births and
deaths, respectively, which can be quantified by the
$m^{\mathrm{th}}$ type sensitivity $s_{j,m}$  to (abundance  and shortage
of) resource $j\in\{1,...,J\}$. Moreover, each resource may be
utilized by all types. If
$\wt \lambda_{j,n}$ denotes the amount which the $n^{\mathrm{th}}$
type uses the resource $j\in\{1,...,J\}$,
then given the population $z^m\in \N^G$ of type $m$
it is reasonable to choose the mean deviations from critical offspring
according to the penalty term
\be{penalty}
   \sum^J_{j=1} s_{j,m}\bigg(R^j-\sum^M_{n=1}
   \wt\lambda_{j,n}z^n\bigg).
\ee
We therefore define the \emph{carrying capacity} by
\be{ccap1}
   K^m
 :=
   \sum^J_{j=1}s_{j,m} R^j,
\ee
the \emph{competition matrix},
$(\lambda_{m,n})_{m,n\in\{1,\dots,M\}}$, by
\be{comp1}
   \lambda_{m,n}
 :=
   \sum^J_{j=1}s_{j,m}\wt\lambda_{j,n},
\ee
and the natural branching rates by
\begin{equation}
\label{brarate}
   \gamma^m.
\end{equation}

The dynamics of the particle system
{{}
\begin{equation}
\label{e:ZG}
   Z^G=\big\{(z^m_{\xi}(t))_{t\ge 0};\,m=1,...,M,\xi\in G\big\}
\end{equation}
}
is given by the following two \emph{independent mechanisms}:
\begin{itemize}
\item[-] {\it Migration. }
Each particle migrates in $G$ according to a continuous time
rate~1 random walk
with transition probabilities $a(\xi,\eta)$ from $\xi$ to $\eta$.
\item[-] {\it State dependent branching. }
{{}
For each $m=1,...,M$ and $\xi\in G$, given the current population
$z_\xi=(z^1_\xi,...,z^M_\xi)$ the following transitions occur:

\begin{itemize}
\item {\em Birth. } Each particle of type $m$ at site $\xi$ gives birth to a new particle at rate
\be{birth}
   \gamma^m\cdot\gamma^m_{\mathrm{birth}}(z_\xi)=\gamma^m\big(\tfrac{1}{2}+K^m\big).
\ee
\item {\em Death. }  Each particle of type $m$ at site $\xi$ dies at rate
\be{death}
   \gamma^m\cdot\gamma^m_{\mathrm{death}}(z_\xi)=\gamma^m\big(\tfrac{1}{2}+\sum_{n=1}^M\lambda_{m,n}z_\xi^m\big).
\ee
\end{itemize}
}
\end{itemize}

{{}Define the interaction function of type $m$ with the environment
and the other populations by setting for all $m\in\{1,...,M\}$, $\xi\in G$, and $y\in \R_+^M$,
\begin{equation}\label{P:eq.003}
   \Gamma^m(y_\xi):=\gamma^m_{\mathrm{birth}}(y_\xi)-\gamma^m_{\mathrm{death}}(y_\xi)
 =
   K^m-\sum_{n=1}^M \lambda_{m,n}y_\xi^n.
\end{equation}}

Some statements we can prove only in special cases. We therefore introduce the following
sets of assumptions.
\begin{asss}[Constant branching rate]
All types reproduce at the same rate, i.e.,
there is a constant $\gamma>0$ such that
\be{Z003}
   \gamma^m =
   \gamma,\quad\mbox{ for all }m\in\{1,...,M\}.
\ee
\label{a:a1}
\end{asss}

\begin{asss}[Type-non-sensitive resources and capacities]
Resources are not sensitive to different types, i.e.,
there is a constant $\lambda\ge 0$ such that
\be{AG2}
   \lambda_{m,n} =
   \lambda,\quad\mbox{ for all }m,n\in\{1,...,M\}.
\ee

\noindent All types
have the same carrying capacity, i.e., there is a constant $K\ge 0$
such that
\be{AG3}
   K^m =
   K,\quad\mbox{ for all }m\in\{1,...,M\}.
\ee
\label{a:a2}
\end{asss}

\begin{definition}[Exchangeable models]
If Assumptions~\ref{a:a1} and~\ref{a:a2} are satisfied,
then the corresponding class of models is called
the exchangeable model.
\label{d:exch}
\end{definition}

We want to characterize the particle model introduced above analytically as a solution of a system of stochastic equations and a martingale problem.
First consider the following system of stochastic equations: for all $m\in\{1,2,...,M\}$, $\xi\in G$, and $t>0$,
\be{e:poisson}
\begin{aligned}
   z^m_{\xi}(t)
 &=
   z^m_{\xi}(0) + \sum_{\eta\in G,\eta\ne\xi} \Big[ \int_{[0,t]\times\R_+}
   \1\big(z^m_{\eta}(s-)\ge u\big)N^{m,\xi}_{\eta}(\d s \, \d u)
  \\
 &\hspace{4cm}{}
  -\int_{[0,t]\times\R_+}\1\big(z^m_{\xi}(s-)\ge u\big)N^{m,\eta}_{\xi}(\d s \, \d u) \Big]\hspace{2em}
  \\
 &{}
   {{}+\int_{[0,t]\times\R_+}\1\big(z^m_{\xi}(s-)\gamma^m_{\mathrm{birth}}(z^m_{\xi}(s-)) \ge u \big)N^{m,+}_{\xi}(\d s \, \d u)}
  \\
 & {}
   -{{}\int_{[0,t]\times\R_+}\1\big(z^m_{\xi}(s-)\gamma^m_{\mathrm{death}}(z^m_{\xi}(s-)) \ge u\big)N^{m,-}_{\xi}(\d s \, \d u).}
    \end{aligned}
    \ee
Here $\{N^{m,\eta}_{\xi}:\xi,\eta\in G,\xi\ne\eta,\,1\le m\le M\}$ are independent
    Poisson processes on $[0,\infty) \times \R_+$ and
    $\{N^{m,+}_{\xi},N^{m,-}_{\xi}:\xi\in G,\,1\le m\le M\}$ are independent
    Poisson processes on {{}$[0,\infty) \times \R_+$}, all independent of
    {{}$Z(0)$}. $N^{m,\eta}_{\xi}$ has intensity measure $a(\xi, \eta) \,\d t \otimes \d u$,
    {{}$N^{m,+}_{\xi}$, $N^{m,-}_{\xi}$} have intensity measure
    {{}$\gamma^m\d t \otimes \d u$} ($\d t$, $\d u$ are Lebesgue measures).

Now we formulate the martingale problem. {{}Define for any countable Abelian group $G$ the domain}
\begin{equation}
\label{ag5}
   {{}\CD(\Omega^G_Z):= \big\{f\in B({\mathcal E}^{\mathrm{par},G}),\Omega^G f\in B({\mathcal E}^{\mathrm{par},G})\big\},}
\ee
where the action of the operator {{}$\Omega^G_{Z}$} acting on {{}$\CD(\Omega^G_Z)$} is given by
\begin{equation}
\label{pregenerator}
\begin{aligned}
   \Omega^G_Z f(z)
 &:=
   \sum_{m=1}^M\sum_{\xi,\eta\in G}z^m_\xi a(\xi,\eta)\big\{f\big(z+e(m,\eta)-e(m,\xi)\big)-f(z)\big\}
  \\
 &
   {{}+\sum_{m=1}^M\gamma^m\sum_{\xi\in G}\gamma^m_{\mathrm{birth}}\big(z_\xi\big)z^m_\xi\big\{f\big(z+e(m,\xi)\big)-f(z)\big\}}
  \\
 &\quad
   {{}+\sum_{m=1}^M\gamma^m\sum_{\xi\in G}\gamma^m_{\mathrm{death}}\big(z_\xi\big)z^m_\xi\Big\{f\big(z-e(m,\xi)\big)-f(z)\Big\},}
\end{aligned}
\end{equation}
and $e{(m,\xi)}\in(\N_0^M)^G$ is defined by $(e(m,\xi))((n,\eta)):=\delta_{(m,\xi),(n,\eta)}$.

{{}Notice that if $G'$ is finite and $z(0)\in{\mathcal E}^{{\rm par},G'}$ (recall from (\ref{LLS})), existence and uniqueness of the solution $Z^{G'}$
with values in $D(\R_+,{\mathcal E}^{\rm par},G')$ of
the $(\Omega_Z,{\mathcal D}(\Omega^{G'}_Z),z(0))$ martingale problem follow by standard theory on jump processes.
Compare also \cite{AtBassPer05} and \cite{AtBarBassPer02}.
Moreover, $Z^{G'}$ is also the unique solution of the system (\ref{e:poisson}) of stochastic equations. }

{{}
Moreover, if $G'\subset G$, $z(0)\in{\mathcal E}^{\mathrm{par},G}$ and $\tilde{Z}^{G'}$ is
a solution of the $(\Omega^{G'}_{Z}, {\mathcal D}(\Omega^{G'}_{Z}),z_{|G'}(0)))$-martingale problem, with the restricted initial state $z_{|G'}(0):=\{z^m_\xi(0);\,\xi\in G'\}$ then
$\tilde{Z}^{G'}$ can be extended to a  ${\mathcal E}^{\mathrm{par},G}$-valued process, $Z^{G',G}$, whose components $\{z^m_\xi,\,\xi\in G\setminus G'\}$ outside of $G'$  are {\em frozen}, i.e.
\begin{equation}
\label{e:ZGLG}
   (z^{G',G})^m_\xi(t):=\left\{\begin{array}{cc}\tilde{z}^m_\xi(t), & \mbox{ if }\xi\in G'\\
   z^m_\xi(0), & \mbox{ if }\xi\in G\setminus G'\end{array}\right.
\end{equation}
}\smallskip

Fix now a countable Abelian group $G$ and approximating finite Abelian groups $\{G_L;\,L\in\mathbb{N}\}$ as in (\ref{ag1}).

\begin{theorem}[Existence of the particle system] 
Fix $z(0)\in{\mathcal E}^{\mathrm{par},G}$. Let for each $L\in\mathbb{N}$, $Z^{G_L,G}$ be the extended unique solution of the
$(\Omega^{G_L}_{Z}, {\mathcal D}(\Omega^{G_L}_{Z}),z_{|G_L}(0))$-martingale problem. Then the following hold:
\begin{itemize}
\item[(i)] The family of processes $\{Z^{G_L,G};\,L\in\mathbb{N}\}$ is tight in $D(\R_+,{\mathcal E}^{\mathrm{par},G})$
for each initial condition $z(0)\in{\mathcal E}^{\mathrm{par},G}$.
\item[(ii)] Any limit point $Z^{G}$ of $\{Z^{G_L,G};\,L\in\mathbb{N}\}$  is  a solution of the
$(\Omega^G_{Z}, {\mathcal D}(\Omega^G_{Z}), z(0))$-martingale problem and a weak solution of (\ref{e:poisson}).
\item[(iii)]
The laws of {(subsequences of)} $\{Z^{G_L,G};\,L\in\mathbb{N}\}$ converge in $D(\R_+,{\mathcal E}^{\mathrm{par},G})$, i.e., there are only limit points $Z^G$
of $\{Z^{G_L,G};\,L\in\mathbb{N}\}$ in $D(\R_+,{\mathcal E}^{\mathrm{par},G})$.
\end{itemize}
\label{T:parti}
\end{theorem}

We prove this theorem in Section \ref{Sub:parti}.


\subsection{A system of interacting branching diffusions}
\label{ss.iabranchdiff}
In this section we introduce the candidate for the continuous mass limit.
Put for any {{}countable Abelian group $G$},
\be{sg2}
   \bar a (\xi,\eta):= a(\eta,\xi), \quad \xi, \eta \in G.
\ee

We consider
the following system of differential equations:
for all $m\in\{1,...,M\}$, and $\xi\in G$,
\begin{equation}
\label{P:eq.007}
\begin{aligned}
   \d x_\xi^m(t)
 &=
   \sum_{\eta\in G} \bar a(\xi,\eta)\big(x^m_\eta(t)-x^m_\xi(t)\big)
   \d t + \gamma^m x^m_\xi(t)\Gamma^m \big(x_\xi(t)\big)\d t \\
 &\quad+
   \sqrt{\gamma^m x_\xi^m(t)}\,\d w_\xi^m(t),
\end{aligned}
\end{equation}
where $(w_\xi^m)_{\xi\in G, m\in\{1,\dots,M\}}$ is a family of independent
standard Brownian motions.

{{}For all $k\in\mathbb{N}_0$, denote by
\begin{equation}
\label{e:Ck}
  C^k(\mathcal{E}^G)
  :=
  \big\{f\in C({\mathcal E}^G):\,f\mbox{ is $k$-times continuously differentiable}\big\}.
\ee
Put $C_b^k(\mathcal{E}^G):=C^k(\mathcal{E}) \cap B(\CE^G)$,
the space of  bounded functions with derivatives up to
the $k$-th order.
{{}Consider}
\be{P:eq.008}
\begin{aligned}
    \Omega^G_X f(x)
  &=
    \sum_{m=1}^{M} \sum_{\xi\in G} \big(\sum_{\eta\in G}
    \bar a(\xi,\eta)\big(x_\eta^m-x_\xi^m\big) + \gamma^m
    x^m_\xi\Gamma^m(x_\xi)\big)\tfrac{\partial f}{\partial x^m_\xi}(x)
   \\
  &+ \tfrac12\sum_{m=1}^{M}\gamma^m\sum_{\xi\in G}x_\xi^m
      \tfrac{\partial^2 f}{(\partial x_\xi^m)^2}(x).
\end{aligned}
\ee
{acting on
\be{e:domainOmegaX}
\begin{aligned}
   {\mathcal D}(\Omega^G_X)
 &:=
   \big\{f\in C_b^3({\mathcal E}^G), \Omega^G_Xf\in B({\mathcal E}^G)
      \big\}.
\end{aligned}
\ee
}


{{}Once more, if $X^{G'}$ solves the $(\Omega^{G'}_X,{\mathcal D}(\Omega^{G'}_X),x'(0))$-martingale problem for some $G'\subset G$, we consider it as an ${\mathcal E}^G$-valued process starting in $x(0)$ such that $(x^{G'})^m_\xi(0)=x^m_\xi(0)$ for all $\xi\in G'$ and $m=1,...,M$, and freezing the dynamics outside of $G'$.}

{{}\begin{remark}[Separating class] Note that the family
\begin{equation}
\label{e:separate}
\begin{aligned}
   {\mathcal H}
   &:=\big\{f_{\mu,\kappa}\big(x\big):=e^{-\sum_{\xi\in G}\mu_\xi\bar{x}_{\xi}}\prod_{\xi\in G}\prod_{m=1}^{M}(x^{m}_{\xi})^{\kappa^{m}_{\xi}};\,\mu\in[0,\infty)^G,\kappa\in\mathbb{N}^G\mbox{ with }\sum_{\xi}\kappa_\xi<\infty,
   \\
   &\hspace{1cm} \mu_\xi>0\mbox{ if }\kappa^m_\xi>0\mbox{ for some }m\in\{1,...,M\},
   \\
   &\hspace{1cm}\mu_\xi>0\mbox{ only for finitely many }\xi\in G
   \big\}
\end{aligned}
\end{equation}
is a subset of ${\mathcal D}(\Omega^G_X)$ which {\em separates points}.
\label{Rem:003}
\end{remark}}

The next theorem states existence and uniqueness of the solution of (\ref{P:eq.007}).
\begin{theorem}[Interacting diffusion well-defined]
Let $G$ and $\{G^L;\,L\in\mathbb{N}\}$ as in (\ref{ag1}),
$X(0)\in\mathcal{E}^G$, and
$\Omega^G_X:{\mathcal D}(\Omega_X)\to C_b(\mathcal{E}^G)$ be as in
(\ref{P:eq.008}).
\begin{itemize}
\item[(i)]
The $(\Omega^G_X,{\mathcal D}(\Omega^G_X),X(0))$-martingale problem has a solution
\begin{equation}
\label{e:XG}
   X^G=\big\{(x^m_{\xi}(t))_{t\ge 0};\,\xi\in G,m=1,...,M\big\}
\end{equation}
with paths in $C(\R_+,\CE^G)$ {that arises as the subsequential limit of the laws of the finite approximations $X^{G_L}$.}

If we are in the exchangeable case, the process $X$ has a unique solution and this solution
is the limit of the finite approximations.
\item[(ii)] Any solution of the system of SDEs given in (\ref{P:eq.007}) solves the $(\Omega^G_X,{\mathcal D}(\Omega^G_X),X(0))$-martingale problem.
\item[(iii)]
Any solution to the $(\Omega^G_X,{\mathcal D}(\Omega^G_X),X(0))$-martingale problem is a weak solution to the
system of SDEs given in (\ref{P:eq.007}).
\end{itemize}
\label{T0}
\end{theorem}

\begin{cor} The system of SDEs given in (\ref{P:eq.007}) {{}has a
weak solution} and this solution has continuous sample paths.
In the exchangeable case it is a strong Markov and Feller process.
\end{cor}

The proof of Theorem~\ref{T0} can be found in Section~\ref{subsection:Tdiff}. The proof uses a diffusion approximation which is the subject of the next section.

\smallskip

If we make stronger assumptions on the initial states we can say more about the states
of the process $X$ later on, and in fact about a {\em proper state space}
of $X$. So let for $p\in[1,\infty)$,
{{}
\begin{equation}
\label{e:lp}
   l^p(\rho)={\mathcal E}^{G}_p:=\big\{x\in (\R^M_+)^G:\,\|x\|_p<\infty\big\},
\end{equation}
where
\begin{equation}
\label{normrhop}
   \|x\|_{p}:=\|x\|_{p,\rho}= \big(\sum_{\xi \in G}\big(\bar{x}_{\xi}\big)^p \rho(\xi)\big)^{\frac{1}{p}}.
\end{equation}
}

Concerning a state space for the multi-type logistic branching diffusion $X$ we have the following result.
\begin{proposition}[State space]
Let $X$ be a solution of (\ref{P:eq.007}) with initial condition $x(0) \in l^p(\rho)$ for some $p \geq 2$.
Then the paths of $X$ lie in the space $l^p(\rho)$, almost surely. In fact, for every $T\geq 0$ we have the bound
\begin{equation}
\label{l^pbound}
   E\big[\sup_{0 \leq t \leq T}\|x(t)\|_{p}\big] < \infty.
\end{equation}
\label{propn:l^pbound}
\end{proposition}

The proof of Proposition~\ref{propn:l^pbound}
is given in Subsection~\ref{ss.momentcalc}.


\subsection{The diffusion limit}
\label{ss.difflimit}
Here we show that the continuous mass population model, $X^G$, can be indeed approximated
by the particle system $Z^G$ using a {\em small
mass, many particles} and {\em rapid reproduction} limit.

For that, consider a family $\big\{Z^{G,\varepsilon};\,\varepsilon>0\}$ with
\begin{equation}
   Z^{G,\varepsilon}=\big\{(z^m_{\varepsilon,\xi}(t))_{t\ge 0},\,m\in\{1,...,M\},\xi\in G\big\}
\label{A0}
\ee
where $Z^{G,\varepsilon}:=\varepsilon\cdot { {}^{\epsilon}Z^{G}}$
and {${}^{\epsilon}Z^G$} is a solution of the $(\Omega^{G}_{Z},{\mathcal D}(\Omega^{G}_{Z}))$-martingale problem in which we rescaled the parameters {{}$\gamma^m\mapsto\frac{\gamma^m}{\varepsilon}$, $\gamma^m_{\mathrm{birth}}\mapsto\tfrac{\gamma^m}{\varepsilon}(\tfrac{1}{2}+\varepsilon K^m)$, $\gamma^m_{\mathrm{death}}(z_\xi)\mapsto\tfrac{\gamma^m}{\varepsilon}(\tfrac{1}{2}+\varepsilon \sum_{n=1}^M\lambda_{m,n}z^n_\xi)$,
(and thus $\Gamma^m\mapsto\varepsilon\Gamma^m$).} That is, in the
$\varepsilon$-approximation the
particles have mass $\varepsilon$, the initial number of particles
is blown up by a factor of $\varepsilon^{-1}$, and the branching is speeded
up by the factor $\varepsilon^{-1}$ and in addition we replace $\Gamma^m$ by $\ve\Gamma^m$ (limit of small perturbation of criticality of the branching). As a consequence $Z^{G,\varepsilon}$ solves the  $(\Omega^{G}_{Z^\varepsilon},{\mathcal D}(\Omega^{G}_{Z^\varepsilon}))$-martingale problem
where
\be{gen3}
\begin{aligned}
    \Omega^{G}_{Z^\varepsilon}f(z)
 &=
    \sum_{m=1}^M\sum_{\xi\in G} \tfrac{z_{\varepsilon,\xi}^m}{\ve} \big\{\sum_{\eta\in G}
    a(\xi,\eta) \big(f(z+{\ve}e(m,\eta)-{\ve}e(m,\xi))-f(z)\big)
    \\
    & \phantom{AAAAAAAAA} + \tfrac{\gamma^m}{\varepsilon}{{}\big(\tfrac{1}{2}+\varepsilon K^m\big)}
    \big(f(z+\vee(m,\xi))-f(z)\big)
    \\
    & \phantom{AAAAAAAAA} +  \tfrac{\gamma^m}{\varepsilon} {{}\big(\tfrac{1}{2}+\varepsilon \sum_{n=1}^M\lambda_{m,n}z^n_\xi\big)}
    \big(f(z-\vee(m,\xi))-f(z)\big)\big\}.
\end{aligned}
\ee

{{}
\begin{theorem}[The diffusion limit]
Let $X(0)$  in ${\mathcal E}^{G}$ be random such that $\Ex[X(0)]\in{\mathcal E}^G$.   Define $Z^\varepsilon(0)$ by letting for all $\varepsilon>0$, $\xi\in G$ and $m\in\{1,2,...,M\}$, $z_{\varepsilon,\xi}^m(0):=\varepsilon\lfloor \tfrac{1}{\varepsilon}x_{\xi}^m(0)\rfloor$, and start all $Z^{G,\varepsilon}$ in $Z^\varepsilon(0)$.
 Then the family $\{Z^{G,\varepsilon};\,\varepsilon>0\}$ is relatively compact and any limit point $X$ satisfies the $(\Omega^G_X,{\mathcal D}(\Omega^G_X),X(0))$-martingale problem.

If the model is exchangeable, the  $(\Omega^G_X,{\mathcal D}(\Omega^G_X),X(0))$-martingale problem has a unique solution, $X$, and
\be{e:difflimit}
   \CL\big[Z^{\varepsilon}\big] \Tvenull \CL \big[X\big],
\ee
where here $\Rightarrow$ means weak convergence in
$D(\R_+,\CE)$ with respect to the Skorohod topology.
\label{Tdiff}
\end{theorem}
}

We shall give the fairly standard proof of Theorem~\ref{Tdiff} in Section~\ref{subsection:Tdiff}.

\subsection{Exponential duality  in the exchangeable model}
\label{ss.expdual}
In this section we focus on the exchangeable models only.
In order to verify uniqueness of the martingale problem of the interacting diffusion process, we shall use the following duality.

Let $X \in {\mathcal E}$
be the weak solution of (\ref{P:eq.007}) in the exchangeable case, i.e.\ the solution of the following system of stochastic
differential equations:
\be{e:self1a}
\begin{aligned}
   &\d x_{\xi}^{m}(t)
   \\
 &=
   \sum_{\eta\in G} \bar a(\xi,\eta)\big(x_\eta^m(t)-x_{\xi}^{m}(t)\big)\d t
   +\gamma x_{\xi}^{m}(t)\big(K-\lambda \bar{x}_{\xi}(t)\big)\d t
   +\sqrt{\gamma x_{\xi}^{m}(t)}\,\d w_{\xi}^{m}(t),
\end{aligned}
\ee
where $\lambda>0$, $K>0$, and the Brownian
motions $\{w^{m}_{\xi};\,m\in\{1,\dots,M\},\xi\in G\}$ are independent.

As we will see this process is dual to the Markov process with state space
\be{Z048}
    \CE^{\rm dual}:=\bigg\{(\alpha,\kappa)\in(\R_+)^G\times(\N_0^M)^G:\sum_{\xi\in G}\alpha_\xi\bar{x}_\xi<\infty, \sum_{\xi\in G}\bar\kappa_\xi\bar{x}_\xi<\infty, \;\forall x\in\CE\bigg\}
\ee
{where
\begin{equation}
\label{e:barkap}
   \bar{\kappa}_{\xi}:=\sum_{m=1}^M\kappa^{m}_{\xi},
\ee}
and with the generator acting on 
{{}
\be{e:thedual3}
\begin{aligned}
    &{\mathcal D}(\Omega_{(\alpha,\kappa)})
    \\
  &:=\big\{f\in B(\CE^{\rm dual}):\,f(\boldsymbol{\cdot},\kappa)\in{\mathcal D}(\Omega_{X}),\,\forall\kappa\in(\N^M_0)^G,f(\alpha,\boldsymbol{\cdot})\in B(\N_0^G),\,\forall\alpha\in{\mathcal E}^G\big\}
\end{aligned}
\ee
}
of the form
\be{e:thedual2}
\begin{aligned}
   &\Omega_{(\alpha,\kappa)}f\big(\alpha,\kappa\big)
  \\
 &:=
   \sum_{m=1}^M\sum_{\xi,\eta\in G} \kappa^{m}_{\xi} \bar a(\xi,\eta)
   \Big(f\big(\alpha,\kappa+e(m,\eta)-e(m,\xi)\big)-f(\alpha,\kappa)\Big)
  \\
 &\qquad+
 \gamma\sum_{m=1}^M\sum_{\xi\in G}{\kappa^{m}_{\xi}\choose 2}
 \Big(f(\alpha,\kappa-e(m,\xi))-f(\alpha,\kappa)\Big)
  \\
 &\qquad+
   \sum_{\xi,\eta\in G} a (\xi,\eta)\big(\alpha_\eta-\alpha_{\xi}\big)\frac{\partial}
   {\partial \alpha_{\xi}}f\big(\alpha,\kappa\big)
   \\
 &\qquad+
   \gamma\sum_{\xi\in G}\alpha_{\xi} \Big(K-\frac12\alpha_{\xi}\Big)\frac{\partial}
   {\partial
\alpha_{\xi}}f\big(\alpha,\kappa\big)
  +\gamma\lambda\sum_{\xi\in G}\alpha_{\xi}\frac{\partial^2}
   {\partial (\alpha_{\xi})^2}f\big(\alpha,\kappa\big),
 \\
 &\qquad+
   \gamma\lambda\sum_{\xi\in G}\bar{\kappa}_{\xi}\frac{\partial}
   {\partial
\alpha_{\xi}}f\big(\alpha,\kappa\big).
\end{aligned}
\ee

This means that the process
$\kappa$ describes an autonomous spatial Kingman coalescent in which particles migrate independently on $G$ according to $a$ and particles at the same site coalesce according to
a Kingman coalescent mechanism with rate $\gamma.$
The process $\alpha$ follows a diffusion of a form analogous to that of
the total mass process $\bar{x}$ with the exception that we have an
additional nonnegative immigration term for $\alpha_{\xi}$ which is given by
$\gamma\lambda\bar{\kappa}_{\xi}.$ This implies that the process
$\alpha$ depends on $\kappa$ whenever $\kappa$ is nonzero.
In this case, $\alpha$ cannot die out completely as the following lemma states.

\begin{lemma}
Let $\alpha(0) \in l^2(\rho)$, $p\ge 2$, and $\kappa(0) \in (\N_0^M)^G$ with $\sum_{\xi \in G} \bar{\kappa}_{\xi} < \infty.$ Then there exists a unique solution for the \label{L.mart}
$(\Omega_{(\alpha,\kappa)},{\mathcal D}(\Omega_{(\alpha,\kappa)}),(\alpha(0),\kappa(0)))$-martingale problem on
$D([0,\infty), (\R_+)^G \times (\N_0^M)^G)$.
Furthermore we have continuous paths with exception of finitely many
jump points in finite time intervals.
\end{lemma}

It turns out that $\CE^{\rm dual}$ can be chosen as a state space for the dual process started a.s.\ in a configuration from $\CE_f^{\rm dual}$ where
\be{Z049}
    \CE_f^{\rm dual}:=\Big\{(\alpha,\kappa)\in(\R_+)^G\times(\N_0^M)^G:(\alpha,\kappa) \mbox{ {has finite support}}\Big\}.
\ee
This is justified by the following lemma whose proof can be found in Section \ref{P:exdual}.

\begin{lemma} \label{ZL1}
    For $(\alpha(0),\kappa(0))\in\CE_f^{\rm dual}$, $(\alpha(t),\kappa(t))\in\CE^{\rm dual}$ a.s.\ for all $t>0$.
\end{lemma}
Consider the duality function
\be{e:self2a}
    H\big((\alpha,\kappa),x\big) :=
   \exp\bigg(-\sum_{\xi\in
 G}\alpha_\xi\bar{x}_{\xi}\bigg) x^{\kappa},
\ee
for $\alpha\in(\R_+)^G$, $x\in((\R_+)^M)^G$ and $\kappa\in(\N_0^M)^G$, where $x^{\kappa}:=\prod_{\xi\in G}\prod_{m=1}^{M}(x^{m}_{\xi})^{\kappa^{m}_{\xi}}$. We then show the following:
\begin{proposition}[A duality for the exchangeable case] \label{L:dual}
    Let $(\alpha(t), \kappa(t))$ be a Markov process with generator $\Omega_{(\alpha,\kappa)}$ defined in (\ref{e:thedual2}) with  $(\alpha(0),\kappa(0))\in\CE_f^{\rm dual}$ such that
    \be{e:thedualinit}
        \Ex\bigg[ \Big(\sum_{\xi \in G} \alpha_\xi(0)\Big)^3\bigg]<\infty\; \mbox{ and }\; n=\sum_{\xi \in G} \bar{\kappa}_{\xi}(0)< \infty.
    \ee
    Let $X$ be a solution to (\ref{e:self1a}), that is independent of $(\alpha, \kappa),$ started from $\bar{X}(0)$ which is bounded above by     a translation invariant $\bar{X}^{inv}$ with  $\Ex[(\bar{x}_{\xi}^{inv})^{n+2}]<\infty.$  Then, for all $t\geq 0,$
    \be{e:thedual}
    \begin{aligned}
       & \Ex\big[H\big((\alpha(0),\kappa(0)),x(t)\big)\big]\\
       & = \Ex\bigg[H\big((\alpha(t),\kappa(t)),x(0)\big) \exp\!\bigg(\!\gamma\!\int^t_0 \bigg\{\sum_{m,\xi}{\kappa^{m}_{\xi}(s)\choose 2}+\sum_{\xi}\big(K-\alpha_{\xi}(s)\big)\bar{\kappa}_{\xi}(s)\bigg\}\d s\bigg)\bigg]
    \end{aligned}
    \ee
\end{proposition}
The proof can be found in Section \ref{P:exdual}. {We will present some applications in Section~\ref{s:twotypecoex}.}

\subsection{Exponential duality and coexistence}
\label{s.dual}

We will now use the duality in order to investigate conditions for
coexistence. In order to prove that there is long-term coexistence we would like to show that for some
$\xi \in G,$
\begin{equation}
\label{s:twotypecoex}
\liminf_{t \rightarrow \infty} \Ex\Big[ x^1_{\xi}(t) \cdot  x^2_{\xi}(t) \Big]>0.
\end{equation}

We see from the duality function that it suffices to consider $\alpha (0) =0$ and $\kappa(0)$ to be the
configuration with a type 1 and a type 2 particle at $\xi$ in order to then show that
\be{AG3d}
\liminf_{t \to \infty} E [H((\alpha(0), x(0)), x(t))] >0.
\ee

From the duality we obtain the following monotonicity property for coexistence in the initial condition.
\begin{proposition}
Let $X^{\theta}$ be started from a constant initial state, namely $x_{\xi}^m(0)=\theta>0$ for all $\xi \in G, m=1,\dots, M.$ If $0<\tilde{\theta}\leq \theta$ and coexistence as in (\ref{s:twotypecoex}) holds for $X^{\theta}$ then it also holds for $X^{\tilde{\theta}}.$ Vice versa, if coexistence 
does not hold for $X^{\theta}$ then it also does not hold for $X^{\tilde{\theta}}$ for any $\tilde{\theta}>\theta.$
\end{proposition}
\begin{proof}
By the duality of Proposition \ref{L:dual} we have that
\begin{equation}
\label{e:byduality}
\begin{aligned}
   &\Ex\big[x^{\tilde{\theta},1}_{\xi}(t) \cdot  x^{\tilde{\theta},2}_{\xi}(t) ]
  \\
 &=
   \Ex\bigg[ \exp(-\tilde{\theta} \bar{\alpha}(t)) \tilde{\theta}^2 \exp\!\bigg(\!\gamma\!\int^t_0 \bigg\{\sum_{\xi}\big(K-\alpha_{\xi}(s)\big)\bar{\kappa}_{\xi}(s)\bigg\}\d s\bigg)\bigg]
  \\
 &\geq
   \frac{\tilde{\theta}^2 }{ \theta^2 }\Ex\bigg[ \exp(-\theta \bar{\alpha}(t)) \theta^2 \exp\!\bigg(\!\gamma\!\int^t_0 \bigg\{\sum_{\xi}\big(K-\alpha_{\xi}(s)\big)\bar{\kappa}_{\xi}(s)\bigg\}\d s\bigg)\bigg]
  \\
 &=
   \frac{\tilde{\theta}^2 }{ \theta^2 }\Ex\big[x^{\theta,1}_{\xi}(t) \cdot  x^{\theta,2}_{\xi}(t) ].
\end{aligned}
\ee
      Taking $\liminf_{t \rightarrow \infty}$ on both sides now implies the statement.
\end{proof}

The following lemma may be helpful to establish coexistence results in forthcoming work.
\begin{lemma}[Total dual mass diverges]
\label{L:alphatoinfinity}
Suppose that the parameters are such that the total mass process $\bar{X}$ started in a translation invariant nontrivial initial condition survives. Then for any $(\alpha, \kappa)$ process with generator (\ref{e:thedual2}) such that $\alpha(0)$ has finite support and $\sum_{\xi \in G} \bar{\kappa}_\xi(0)\geq1$
we obtain
\begin{equation}
\label{e:alphatoinfinity}
    \Pr\bigg(\sum_{\xi \in G} \alpha_\xi(t) \mathop{\longrightarrow}\limits_{t\to\infty} \infty\bigg)=1.
\end{equation}
\end{lemma}

\begin{proof}
Let $\bar{X}$ be started from the constant configuration $\bar{x}_\xi(0)=1$ for all $\xi \in G.$
We let $(\alpha^0,\kappa^0)$ be the process with generator (\ref{e:thedual2}) started from $(\alpha(0),0)$. Then by duality, for $t \geq 0,$
\begin{equation}
\label{s:dualalpha0}
\Ex\bigg[ \exp\bigg( - \sum_{\xi \in G} \alpha_{\xi}^0(t) \bigg)\bigg]  =\Ex\bigg[ \exp\bigg( -\sum_{\xi \in G}\alpha_{\xi}^0(0)  \bar{x}_\xi(t)\bigg) \bigg].
\end{equation}
By Lemma 8.1 of \cite{HW07} we have
\begin{equation}
\label{s:extgrowth}
    \Pr\Bigg(  \sum_{\xi \in G} \alpha^0_\xi(t) \mathop{\longrightarrow}\limits_{t\to\infty} \infty\quad \text{or}\quad \exists t'< \infty: \sum_{\xi \in G} \alpha^0_\xi(t)=0 \, \,\forall t\geq t'\Bigg)=1.
\end{equation}
On the other hand, by Theorem~5 of \cite{HW07} we have
$\bar{X}(t) \Rightarrow \bar{X}(\infty)$ for $t \rightarrow \infty,$ where $\bar{X}(\infty)$
is translation invariant and also nontrivial by our assumption of survival.
 Thus, letting $t \rightarrow \infty$ in (\ref{s:dualalpha0}) we obtain with
(\ref{s:extgrowth})
\be{Z050}
\begin{aligned}
    & 1-\Pr\bigg(  \sum_{\xi \in G} \alpha^0_\xi(t) \mathop{\longrightarrow}\limits_{t\to\infty} \infty \bigg)\\
    &= \Pr\bigg(\exists t'< \infty: \sum_{\xi \in G} \alpha^0_\xi(t)=0 \,\, \forall t\geq t'\bigg) =\Ex\bigg[ \exp\bigg( -\sum_{\xi \in G}\alpha_{\xi}(0)  \bar{x}_\xi(\infty)\bigg) \bigg].
\end{aligned}
\ee
This implies that for any $\varepsilon>0,$
\begin{equation}
\label{s:growthlowbound0}
\inf_{\alpha(0)\text{ s.t.\ } \exists \xi: \alpha_{\xi}(0)\geq  \varepsilon}
\Pr\bigg(  \sum_{\xi \in G} \alpha^0_\xi(t) \mathop{\longrightarrow}\limits_{t\to\infty} \infty \bigg)\geq 1-\Ex\big[e^{ -\varepsilon \bar{x}_{0}(\infty)}\big]>0.
\end{equation}
Note that we may apply Proposition \ref{monotonicity} to $\alpha$ as the migration and drift terms fulfill
the assumptions. Thus, due to the monotonicity in the immigration term stated there  we obtain also for $(\alpha, \kappa)$ with any $\kappa(0)$ that
\begin{equation}
\label{s:growthlowbound}
\inf_{\alpha(0)\text{ s.t.\ }\exists \xi : \alpha_{\xi}(0)\ge \varepsilon}
\Pr\bigg(  \sum_{\xi \in G} \alpha_\xi(t) \mathop{\longrightarrow}\limits_{t\to\infty} \infty \bigg)>0.
\end{equation}
It now remains to show that there exists an $\varepsilon>0$ such that for any $(\alpha, \kappa)$ process
\begin{equation}
\label{s:epsgrowth}
\inf_{(\alpha(0), \kappa(0))\text{ s.t.\ }  \sum_{\xi \in G} \bar{\kappa}_\xi(0)\geq1}
\Pr\big( \exists t< \infty, \,\xi \in G: \alpha_\xi(t)\geq \varepsilon\big)>0.
\end{equation}
Note that if $\sum_{\xi \in G} \bar{\kappa}_\xi(0)\geq1$ then there exists $\zeta \in G$ such that
$\bar{\kappa}_{\zeta}(t)\geq 1$ at least for $t \leq T \sim \exp(1).$ Let $0<\delta<\gamma \lambda$ be arbitrary.
Then there exists an $\varepsilon>0$ such that for any $\alpha \in \R_+^{\Z^d}$ with $0 \leq \alpha_\zeta\leq \varepsilon,$
\begin{equation}
\label{s:driftbound}
\sum_{\eta \in G} \bar{a}(\zeta,\eta) (\alpha_\eta - \alpha_\zeta)  + \gamma \alpha_\zeta \Big(K-\frac{1}{2} \alpha_\zeta\Big)+\gamma \lambda \geq \gamma \lambda-\delta>0.
\end{equation}
Since the left hand side of (\ref{s:driftbound}) is a lower bound for the drift term of $\alpha_{\zeta}$ up to time $T\wedge S,$ where $S=\inf\{ t\geq 0: \alpha_{\zeta}(t) \geq \varepsilon\}$ this implies that according to Proposition \ref{monotonicity} we can couple $\alpha_\zeta$ to $\tilde{\alpha}$ which solves
\begin{equation}
\label{s:tildealphaxi}
\d \tilde{\alpha}(t) = ( \gamma \lambda-\delta )\,\d t
+ \sqrt{\gamma \lambda \tilde{\alpha}(t) }\,\d w_t
\end{equation}
such that $\tilde{\alpha}(0)= \alpha_{\zeta}(0)$ and $\tilde{\alpha}(t)\leq \alpha_{\zeta}(t)$ for $t \leq T \wedge S.$  Setting also
$\tilde{S}=\inf\{ t\geq 0: \tilde{\alpha}(t) \geq \varepsilon\}$  it follows immediately that $S\leq \tilde{S} < \infty$ a.s. and in particular that
\be{Z051}
\inf_{(\alpha(0), \kappa(0))\text{ s.t.\ }  \sum_{\xi \in G} \bar{\kappa}_\xi(0)\geq1}
\Pr\big( \exists t< \infty, \,\xi \in G: \alpha_\xi(t)\geq \varepsilon\big)\geq \Pr\big(\tilde{S}\leq T\big)>0,
\ee
which shows (\ref{s:epsgrowth}). Taking (\ref{s:growthlowbound}) and  (\ref{s:epsgrowth}) together with the Markov property of $(\alpha, \kappa)$ now implies that for all $a \in \R,$
\begin{equation}
\inf_{(\alpha(0), \kappa(0))\text{ s.t.\ }  \sum_{\xi \in G} \bar{\kappa}_\xi(0)\geq1}
\Pr\bigg(\exists t'< \infty: \sum_{\xi \in G} \alpha_\xi(t)\geq a \, \forall t\geq t'\bigg)=:\varepsilon>0.
\end{equation}
Therefore, for initial conditions with
$\sum_{\xi \in G} \bar{\kappa}_\xi(0)\geq1$ and so $\sum_{\xi \in G} \bar{\kappa}_\xi(s)\geq1$ for all
$s\geq 0,$ we obtain  by  the Markov property and martingale convergence that
\be{Z052}
\begin{aligned}
    \varepsilon
    &\leq \Pr\bigg(\exists t'< \infty: \sum_{\xi \in G} \alpha_\xi(t)\geq a \, \forall t\geq t'\Big|\big(\alpha(s), \kappa(s)\big)\bigg)\\
    &= \Pr\bigg(\exists t'< \infty:  \sum_{\xi \in G} \alpha_\xi(t)\geq a \, \forall t\geq t'\Big| {\cal F}_{s}\bigg)\\
    &\rightarrow 1_{\{\exists t'< \infty  \text{ s.t. }  \sum_{\xi \in G} \alpha_\xi(t)\geq a \, \forall t\geq t'\}} \quad \text{a.s.}
\end{aligned}
\ee
as $s \rightarrow \infty.$ This implies the result.
\end{proof}

}

\section{Proofs of Theorem~\ref{T:parti},~\ref{T0} and~\ref{Tdiff}}
\label{s.proofcc}


{{}Fix $G$ and $\{G_L;\,L\in\mathbb{N}\}$ as in (\ref{ag1}).
Recall $Z^{G_L,G,\varepsilon}$ and $X$ from (\ref{e:ZGLG}), (\ref{gen3}) and (\ref{P:eq.007}).
In this section we give the proofs of  Theorems~\ref{T:parti},~\ref{T0} and~\ref{Tdiff}.
In particular, we verify weak convergence of $Z^{G_L,G,\varepsilon}$  along a subsequence as
first $L\to\infty$ and then $\varepsilon\to 0$, and first $\varepsilon\to 0$ and then $L\to\infty$,
and show that in both rescaling regimes possible limit points agree.

In Subsection~\ref{Sub:parti} we give bounds on the first moments of the supremum of a component of $Z^{G_L,G,\varepsilon}$
which are uniform in $L\in\mathbb{N}$ and $\varepsilon>0$. We apply them with $\varepsilon=1$ to verify tightness of the family indexed by $L\in\mathbb{N}$. In Subsection~\ref{subsection:Tdiff} we first verify tightness of a family of  rescaled limit points
indexed by $\varepsilon>0$.  We also show that any limit point is a weak solution of (\ref{P:eq.007}).
}

\subsection{First moment bounds and proof of Theorem~\ref{T:parti}}
\label{Sub:parti}
The main goal of this subsection is to prove  Theorem~\ref{T:parti}.
{{}Recall $\{G_L;\,L\in\mathbb{N}\}$ and $G$ from (\ref{ag1}), the spaces ${\mathcal E}^{\mathrm{par},G}$ from (\ref{LLS}), the unique  solution $Z^{G_L}$ of the $(\Omega_Z^{G_L},{\mathcal D}(\Omega_Z^{G_L}))$-martingale problem from (\ref{ag5}), (\ref{pregenerator}), and (\ref{e:ZG}), as well as its extension to a $(\R_+^M)^G$-valued process,  $Z^{G_L,G}$, obtained by freezing all components outside of $G_L$.

We start by showing the tightness claimed in (i) of Theorem~\ref{T:parti}. 
Here, we want to apply Lemma~4.5.1 combined with Remark~4.5.2 in \cite{EthierKurtz1986}.
We therefore need to verify the compact containment condition and uniform convergence of generators. As a preparation we verify  moment bounds. We proceed in three steps.}\smallskip

\noindent
{\em Step~1 (Uniform first moment bounds) }
Fix $\varepsilon>0$, and recall from (\ref{A0}) the re-scaled particle process
\begin{equation}
\label{e:resca}
   Z^{G_L,G,\varepsilon}=\big\{z^m_{\varepsilon,L,\xi};\,m\in\{1,...,M\},\xi\in G\big\}
\end{equation}
obtained by assigning particles individual mass $\varepsilon$, blowing up the initial number of particles by a factor $\varepsilon^{-1}$, speeding up the branching rate by a factor $\varepsilon^{-1}$, and letting $\Gamma^m:=\varepsilon\Gamma^m$.

The following applies to $Z^{G_L,G}$ if we let $\varepsilon=1$, but will be applied with a general $\varepsilon>0$ in the next section.

\begin{lemma}[First moment bounds]
Let $Z(0)$ be a random element in ${\mathcal E}^{\mathrm{par},G}$ such that $\sum_{\xi\in G}\rho(\xi)\Ex[\bar{z}_{\xi}(0)]<\infty$.   Define $Z^\varepsilon(0)$ by letting for all $\varepsilon>0$, $\xi\in G$ and $m\in\{1,2,...,M\}$, $z_{\varepsilon,\xi}^m(0):=\varepsilon\lfloor \tfrac{1}{\varepsilon}z_{\xi}^m(0)\rfloor$, and start all $Z^{G_L,G,\varepsilon}$ in $Z^\varepsilon(0)$.
Then the following hold:
    \begin{itemize}
    \item[(i)]
    For for each $1\le m\le M$, $T>0$ and $\varepsilon>0$, there is a constant $C(T,\varepsilon)$ such that
 \begin{equation}
    \label{ee:moment4-11}
    \begin{aligned}
       \sup_{L\in\N}\sum_{\xi\in
        G}\rho(\xi)\E \big[\sup_{0 \leq t \leq T}(z_{\varepsilon,L,\xi}^{m}(t))\big]
      \leq C(T,\varepsilon)
        \sum_{\xi\in G}\rho(\xi)\Ex\big[\bar{z}_{\xi}(0)\big].
    \end{aligned}
  \end{equation}
 \item[(ii)]
For each $1\le m\le M$ and $T>0$, there is a constant $C(T)$ such that
 \begin{equation}
    \label{ee:moment4-1}
    \begin{aligned}
      \sup_{\varepsilon>0}  \sup_{L\in\N}\sum_{\xi\in
        G}\rho(\xi)\E \big[\sup_{0 \leq t \leq T}(z_{\varepsilon,L,\xi}^{m}(t))\big]
      \leq C(T)
        \sum_{\xi\in G}\rho(\xi)\Ex\big[\big(z_{\xi}^m(0)\big)^2\big].
    \end{aligned}
  \end{equation}
  \end{itemize}
  \label{L:moments}
\end{lemma}

\begin{proof} 
~
~
{{}Applying the generator $\Omega^{G_L}_Z$ (recall from (\ref{gen3})) to $f^p_{m_0,\xi_0}(z)=(z^{m_0}_{\xi_0})^p$, $p\in\{1,2\}$, yields
for all $\varepsilon>0$ and for $\xi_0 \in G_L$, and $m_0\in\{1,...,M\}$,
\begin{equation}
\label{e:moment0p}
\begin{aligned}
   &\Omega^{G_L,\varepsilon}_Zf^p_{m_0,\xi_0}(z)
   \\
   &=\sum_{m=1}^M\sum_{\xi\in G} \tfrac{z_{\varepsilon,L,\xi}^m}{\ve} \big\{\sum_{\eta\in G}
    a(\xi,\eta) \big((z^{m_0}_{\varepsilon,L,\xi_0}+{\ve}e(m,\eta)(m_0,\xi_0)-{\ve}e(m,\xi)(m_0,\xi_0))^p-
    (z^{m_0}_{\varepsilon,L,\xi_0})^p\big)
    \\
    & \phantom{AAAAAAAAA} + \tfrac{\gamma^m}{\varepsilon}{{}\big(\tfrac{1}{2}+\varepsilon K^m\big)}
    \big((z^{m_0}_{\varepsilon,L,\xi_0}+{\ve}e(m,\xi)(m_0,\xi_0))^p-(z^{m_0}_{\varepsilon,L,\xi_0})^p\big)
    \\
    & \phantom{AAAAAAAAA} +  \tfrac{\gamma^m}{\varepsilon} {{}\big(\tfrac{1}{2}+\varepsilon \sum_{n=1}^M\lambda_{m,n}z^n_{\varepsilon,L,\xi}\big)}
    \big((z^{m_0}_{\varepsilon,L,\xi_0}-{\ve}e(m,\xi)(m_0,\xi_0))^p-(z^{m_0}_{\varepsilon,L,\xi_0})^p\big)\big\}.
  \end{aligned}
\end{equation}

Thus
 \begin{equation}
\label{e:moment0p1}
\begin{aligned}
   &\Omega^{G_L,\varepsilon}_Zf^1_{m_0,\xi_0}(z)
   \\
   &=
   \sum_{m=1}^M\sum_{\xi\in G} z_{\varepsilon,L,\xi}^m \big\{\sum_{\eta\in G}
    a(\xi,\eta) \big(e(m,\eta)(m_0,\xi_0)-e(m,\xi)(m_0,\xi_0)\big)
    \\
    & \phantom{AAAAAAAAA} + \gamma^m{{}\Gamma^m(z_{\varepsilon,L,\xi})}e(m,\xi)(m_0,\xi_0)\big\}
   \\
   &=\sum_{\eta\in G_L}z^{m_0}_{\varepsilon,L,\eta}\big(\bar{a}(\xi_0,\eta)-\delta(\xi_0,\eta)\big)
    +\gamma^{m_0}\Gamma^{m_0}\big(z_{\varepsilon,L,\xi_0}\big)\cdot z^{m_0}_{\varepsilon,L,\xi_0}
    \\
    &\le
    \sum_{\eta\in G_L}z^{m_0}_{\varepsilon,L,\eta}\big(\bar{a}(\xi_0,\eta)-\delta(\xi_0,\eta)\big)
    +\gamma^{m_0}K^{m_0}\cdot z^{m_0}_{\varepsilon,L,\xi_0}
 \end{aligned}
\end{equation}
and
\begin{equation}
\label{e:moment0p2}
\begin{aligned}
   &\Omega^{G_L,\varepsilon}_Zf^2_{m_0,\xi_0}(z)
   \\
   &=\sum_{m=1}^M\sum_{\xi\in G} z_{\varepsilon,L,\xi}^m \Big\{\sum_{\eta\in G}
    a(\xi,\eta)\big(e(m,\eta)(m_0,\xi_0)-e(m,\xi)(m_0,\xi_0)\big) 2(z^{m_0}_{\varepsilon,L,\xi_0})
    \\
    & \phantom{AAAAAAAAA} +\varepsilon\sum_{\eta\in G}
    a(\xi,\eta)\big(e(m,\eta)(m_0,\xi_0)-e(m,\xi)(m_0,\xi_0)\big)^2\Big\}
    \\
    & \;+2\gamma^{m_0}{{}\Gamma^{m_0}\big(z_{\varepsilon,L,\xi_0}\big)}\big(z^{m_0}_{\varepsilon,L,\xi_0}\big)^2
    +  \gamma^{m_0}{{}\big(1+\varepsilon K^{m_0}+\varepsilon \sum_{n=1}^M\lambda_{m_0,n}z^n_{\varepsilon,L,\xi_0}\big)}
    z^{m_0}_{\varepsilon,L,\xi_0}
    \\
    &=2\sum_{\eta\in G_L}\big(z^{m_0}_{\varepsilon,L,\eta}\big)\big(z^{m_0}_{\varepsilon,L,\xi_0}\big)\big(\bar{a}(\xi_0,\eta)-\delta(\xi_0,\eta)\big)
    +2\gamma^{m_0}\Gamma^{m_0}\big(z_{\varepsilon,L,\xi_0}\big)\cdot \big(z^{m_0}_{\varepsilon,L,\xi_0}\big)^2
    \\
    &\;+\varepsilon\sum_{m=1}^M\sum_{\xi\in G} z_{\varepsilon,L,\xi}^m \sum_{\eta\in G}
    a(\xi,\eta)\big(e(m,\eta)(m_0,\xi_0)-e(m,\xi)(m_0,\xi_0)\big)^2
    \\
    &\;+ \gamma^{m_0}{{}\big(1+\varepsilon K^{m_0}+\varepsilon \sum_{n=1}^M\lambda_{m_0,n}z^n_{\varepsilon,L,\xi_0}\big)}
    z^{m_0}_{\varepsilon,L,\xi_0}.
  \end{aligned}
\end{equation}

Put
\begin{equation}
\label{overline}
   \overline{\gamma}:=\max_{1\le m\le M}\gamma^m;\;\overline{K}:=\max_{1\le m\le M}K^m;\;\overline{\lambda}:=\max_{1\le m,n\le M}\lambda_{m,n}.
\end{equation}

    Then
    \begin{equation}
\label{e:moment1p1}
\begin{aligned}
   \Omega^{G_L,\varepsilon}_Zf^1_{m_0,\xi_0}(z)
 &\le
    \sum_{\eta\in G_L}z^{m_0}_{\varepsilon,L,\eta}\big(\bar{a}(\xi_0,\eta)-\delta(\xi_0,\eta)\big)
    +\overline{\gamma}\overline{K}\cdot z^{m_0}_{\varepsilon,L,\xi_0}
 \end{aligned}
\end{equation}
and
\begin{equation}
\label{e:moment1p2}
\begin{aligned}
   &\Omega^{G_L,\varepsilon}_Zf^2_{m_0,\xi_0}(z)
   \\
    &\le \sum_{\eta\in G_L}\big(z^{m_0}_{\varepsilon,L,\eta}\big)^2\big(\bar{a}(\xi_0,\eta)-\delta(\xi_0,\eta)\big)
    +2\overline{\gamma}\overline{K}\big(z^{m_0}_{\varepsilon,L,\xi_0}\big)^2
    \\
    &\;+\varepsilon z_{\varepsilon,L,\xi_0}^{m_0}+2\varepsilon\sum_{\xi\in G}\hat a(\xi_0,\xi) z_{\varepsilon,L,\xi}^{m_0}
    +\overline{\gamma}{{}\big(1+\varepsilon \overline{K}+\varepsilon \overline{\lambda} \bar{z}_{\varepsilon,L,\xi_0}\big)}
    z^{m_0}_{\varepsilon,L,\xi_0}.
  \end{aligned}
\end{equation}

{{}Notice that neither $f^p_{m_0,\xi_0}$ nor $\Omega^{G_L,\varepsilon}_Zf^p_{m_0,\xi_0}$ are bounded functions.  However, we can make all the coming arguments work by replacing $f^p_{m_0,\xi_0}$ by
$\tilde{f}^p_{m_0,\xi_0}:={f}^p_{m_0,\xi_0}\cdot e^{-\mu z_{\xi_0}^{m_0}}$ and then use monotone convergence (as $\mu\downarrow 0$). Such calculations are quite involved but standard, so we omit them here and rather work here directly with $f^p_{m_0,\xi_0}$.}

}

Thus  for $\xi_0\in G_L$, and $m_0\in\{1,...,M\}$,
\begin{equation}
\label{e:moment}
\begin{aligned}
   \tfrac{\d}{\d t}\Ex\big[z_{\varepsilon,L,\xi_0}^{m_0}(t)\big]
 &\le
   \sum_{\eta\in G_L}\Ex\big[z^{m_0}_{\varepsilon,L,\eta}(t)\big]\big(\bar a(\xi_0,\eta)-\delta(\xi_0,\eta)\big)
   +
  {{}\overline{\gamma}\overline{K}\Ex\big[z^{m_0}_{\varepsilon,L,\xi_0}(t)\big],}
\end{aligned}
\end{equation}
while $\frac{\d}{\d t}\Ex[z_{\varepsilon,L,\xi_0}^{m_0}(t)]=0$
for $\xi_0 \in G\setminus G_L$.

Put
\begin{equation}\label{eq.gamma}
   C
 :=
   \overline{\gamma}\overline{K}<\infty.
\end{equation}
Then for $\xi_0\in G_L$,
\begin{equation}
\label{e:moment2}
\begin{aligned}
 \frac{\d}{\d t}\Ex[z_{\varepsilon,L,\xi_0}^{m_0}(t)]
&<
   \sum_{\eta\in G_L}\Ex\big[z^{m_0}_{\varepsilon,L,\eta}(t)\big]\big(\bar a(\xi_0,\eta)-\delta(\xi_0,\eta)\big)
   +
  \big(1+C\big)\Ex\big[z^{m_0}_{\varepsilon,L,\xi_0}(t)\big],
\end{aligned}
\end{equation}
 and the right hand side of (\ref{e:moment2}) is non-negative for all $\xi_0\in G$.  Consequently,  for all $\xi_0\in G_L$,
\begin{equation}
\label{e:moment3}
\begin{aligned}
   \Ex\big[z_{\varepsilon,L,\xi_0}^{m_0}(t)\big]
 &\le
   e^{(C+1) t}\sum_{\eta\in G}a_t(\xi_0,\eta)\Ex\big[z^{m_0}_{\varepsilon,L,\eta}(0)\big].
\end{aligned}
\end{equation}

In particular, by (\ref{Gr3}), for all $L\in\N$,
\begin{equation}
\label{e:moment4}
\begin{aligned}
   \sum_{\xi\in G}\rho(\xi)\Ex\big[z_{\varepsilon,L,\xi}^{m}(t)\big]
 &\le
   e^{(C+1) t}\sum_{\xi\in G}\rho(\xi)\Ex\big[z^{m}_{\varepsilon,L,\xi}(0)\big] = e^{(C+1) t}\sum_{\xi\in G}\rho(\xi)\Ex\big[z^{m}_{\xi}(0)\big].
\end{aligned}
\end{equation}

Moreover,
for $\xi_0\in G_L$, and $m_0\in\{1,...,M\}$,
\begin{equation}
\label{e:moment2}
\begin{aligned}
 &\frac{\d}{\d t}\Ex\big[\big(\bar{z}_{\varepsilon,L,\xi_0}(t)\big)^2\big]
 \\
&\le
   \sum_{\eta\in G_L}\Ex\big[\big(\bar{z}_{\varepsilon,L,\eta}(t)\big)^2\big]\big(\bar a(\xi_0,\eta)-\delta(\xi_0,\eta)\big)
   +
  \big(2C+\varepsilon^2\overline{\gamma}\overline{\lambda}\big)\Ex\big[\big(\bar{z}_{\varepsilon,L,\xi_0}(t)\big)^2\big]
  \\
  &+\varepsilon\sum_{\xi\in G}\hat a(\xi_0,\xi) \Ex\big[z_{\varepsilon,L,\xi}^{m_0}(t)\big]
    +\varepsilon {{}\big(1+\overline{\gamma}+\varepsilon C\big)}
    \Ex\big[z^{m_0}_{\varepsilon,L,\xi_0}(t)\big].
\end{aligned}
\end{equation}

It is standard to conclude from here - using (\ref{e:moment4}) that we can find a constant $\tilde{C}<\infty$  such that
by (\ref{Gr3}), for all $L\in\N$,
\begin{equation}
\label{e:moment4p2}
\begin{aligned}
   \sum_{\xi\in G}\rho(\xi)\Ex\big[\big(\bar{z}_{\varepsilon,L,\xi}(t)\big)^2\big]
 &\le
   e^{\tilde{C} t}\sum_{\xi\in G}\rho(\xi)\Ex\big[\big(z_{\varepsilon,L,\xi}(0)\big)^2\big]
   \\
 &\le
   e^{\tilde{C} t}\sum_{\xi\in G}\rho(\xi)\Ex\big[\big(\bar{z}_{\xi}(0)\big)^2\big]<\infty.
\end{aligned}
\end{equation}


We will now use (\ref{e:moment4}) to get the stronger result stated in the lemma.
 The process {{}$Z^{G_L,\varepsilon}$} with initial condition {{}$Z^{G_L,\varepsilon}(0) \in \CE^{{\rm par},\varepsilon,G_L}$} can be
    constructed as the unique solution to
    \be{e:007}
    \begin{aligned}
        &z^m_{\varepsilon,L,\xi}(t)
        \\
        & =  z^m_{\varepsilon,L,\xi}(0) + \sum_{\eta\in G_L,\eta\ne\xi} \Big[ \int_{[0,t]\times\R_+}
        \varepsilon \1\big(z^m_{\varepsilon,L,\eta}(s-)\ge \varepsilon u\big)N^{m,\xi}_{L,\eta}(\d s \, \d u) \\
        & \hspace{4cm}{} - \int_{[0,t]\times\R_+} \varepsilon \1\big(z^m_{L,\xi}(s-) \ge  \varepsilon u\big)N^{m,\eta}_{L,\xi}(\d s \, \d u) \Big]\hspace{2em} \\
        & {} + \int_{[0,{{}t}]\times\R_+} \varepsilon \1\big({{}z^m_{\varepsilon,L,\xi}(s-)\big(\tfrac{1}{2}+\varepsilon K^m\big) \ge\varepsilon  u}\big)N^{m,+}_{{{}\varepsilon},L,\xi}(\d s \, \d u) \\
        & {} - \int_{[0,{{}t}]\times\R_+} \varepsilon \1\big({{}z^m_{\varepsilon,L,\xi}(s-)\big(\tfrac{1}{2}+\sum_{n=1}^M\lambda_{m,n}z^n_{\varepsilon,L,\xi}(s-)\big) \ge \varepsilon u}\big)N^{m,-}_{{{}\varepsilon},L,\xi}(\d s \, \d u)
    \end{aligned}
\ee
for all $\xi\in G_L$, and $t \ge 0$. Here $\{N^{m,\eta}_{L,\xi}:\xi,\eta\in G_L,\xi\ne\eta,\,1\le m\le M\}$ are independent
    Poisson processes on $[0,\infty) \times \R_+$ and
    $\{N^{m,+}_{{{}\varepsilon},L,\xi},N^{m,-}_{{{}\varepsilon},L,\xi}:\xi\in G_L,\,1\le m\le M\}$ are independent
    Poisson processes on $[0,\infty) \times \R_+$, all independent of
    {{}$Z(0)$}. $N^{m,\eta}_{L,\xi}$ has intensity measure $a(\eta,\xi) \,\d t \otimes \d u$,
    $N^{m,+}_{{{}\varepsilon},L,\xi}$, $N^{m,-}_{{{}\varepsilon},L,\xi}$ have intensity measure
    $(\gamma^m{{}\varepsilon})\, \d t \otimes \d u$ ($\d t$, $\d u$ are Lebesgue measures).
    For fixed {{}$Z(0)$},
    {{}$(Z^{\varepsilon,G_L}(t))$} is adapted to the filtration generated by these
    Poisson processes.

(i) Hence
\begin{equation}
    \label{e:005a}
    \begin{aligned}
       &\Ex\big[\sup_{0\le t\le T}z^m_{\varepsilon,L,\xi}(t)\big]
       \\
     &\le
       \Ex\big[z^m_\xi(0)\big]+
       \int^T_0 \sum_{\eta\in G_L}a(\xi,\eta)\Ex\big[z^m_{\varepsilon,L,\eta}(s)\big]\d s
      +
       {{}\tfrac{\gamma^m}{\varepsilon}} \int^{{{}T}}_0\Ex\big[z^m_{\varepsilon,L,\xi}(s)\big]\d s.
    \end{aligned}
    \end{equation}

    Then    by (\ref{e:moment4}),
    \begin{equation}
    \label{e:002}
    \begin{aligned}
       &\sum_{\xi\in G_L}\rho(\xi)\Ex\big[\sup_{0\le t\le T}\bar{z}_{\varepsilon, L,\xi}(t)\big]
       \\
       &\le
       \sum_{\xi\in G_L}\rho(\xi)\Ex[\bar{z}_\xi(0)]+R
       \int^T_0 \sum_{\eta\in G_L}\rho(\eta)\Ex\big[\bar{z}_{\varepsilon,L,\eta}(s)\big]\d s
       + \tfrac{\overline{\gamma}}{\varepsilon}\int^{T}_0 \sum_{\xi\in G_L}\rho(\xi)\Ex\big[\bar{z}_{\varepsilon,L,\xi}(s)\big]\d s
       \end{aligned}
    \end{equation}

    We therefore can find $C(T,\varepsilon)<\infty$ such that for all $L\in\mathbb{N}$,
    \begin{equation}
    \label{e:009}
       \sum_{\xi\in G}\rho(\xi) \Ex\big[\sup_{0\le t\le T}\bar{z}_{ \varepsilon,L,\xi}(t)\big]
       \le
      C(T,\varepsilon)\sum_{\xi\in G}\rho(\xi) \Ex[\bar{z}_\xi(0)],
    \end{equation}
    which proves (\ref{ee:moment4-1}).\smallskip

(ii)    To get a bound uniform in $\varepsilon$, we need to take cancellations due to birth and death into account, i.e.,
 \be{e:007supii}
    \begin{aligned}
        &\sup_{t\in[0,T]}z^m_{\varepsilon,L,\xi}(t)
        \\
        & \le  z^m_{\varepsilon,L,\xi}(0) + \sum_{\eta\in G_L,\eta\ne\xi} \Big[ \int_{[0,T]\times\R_+}
        \varepsilon \1\big(z^m_{\varepsilon,L,\eta}(s-)\ge \varepsilon u\big)N^{m,\xi}_{L,\eta}(\d s \, \d u) \\
        &+
        \tfrac{\gamma^m}{\varepsilon}\int^{{{}T}}_0 {{}\varepsilon\Gamma^m\big(z_{\varepsilon,L,\xi}(s)\big)}\cdot z^m_{\varepsilon,L,\xi}(s)
        \d s
        +\sup_{t\in[0,T]}\big|M^{G_L,\varepsilon}(t)\big|,
\end{aligned}
\ee
where $(M^{G_L,\varepsilon}:=M^{G_L,\varepsilon}(t))_{t\ge 0}$ defined by
\begin{equation}
\label{e:compense}
\begin{aligned}
   &M^{G_L,\varepsilon}(t)
     \\
   &:= \int_{[0,{{}t}]\times\R_+} \varepsilon \1\big(z^m_{\varepsilon,L,\xi}(s-)\big(\tfrac{1}{2}+\varepsilon K^m\big) \ge\varepsilon  u\big)N^{m,+}_{{{}\varepsilon},L,\xi}(\d s \, \d u )
   \\
   &\;-\int_{[0,{{}t}]\times\R_+} \varepsilon \1\big(z^m_{\varepsilon,L,\xi}(s-)\big(\tfrac{1}{2}+\varepsilon \sum_{n=1}^M\lambda_{n,m}z^m_{\varepsilon,L,\xi}(s-)\big) \ge\varepsilon  u\big)N^{m,+}_{{{}\varepsilon},L,\xi}(\d s \, \d u )
   \\
   &\;-\tfrac{\gamma^m}{\varepsilon}\int^{{{}t}}_0 \varepsilon\Gamma^m\big(z_{\varepsilon,L,\xi}(s)\big)\cdot z^m_{\varepsilon,L,\xi}(s)
        \d s
\end{aligned}
\ee
is a (local) martingale (compare, \cite[Lemma 2.1]{BirknerZaehle2007}).\smallskip

We obtain
\begin{equation}
    \label{e:005ii}
    \begin{aligned}
       &\Ex\big[\sup_{0\le t\le T}z^m_{\varepsilon,L,\xi}(t)\big]
       \\
     &\le
       \Ex\big[z^m_\xi(0)\big]+
       \int^T_0 \sum_{\eta\in G_L}a(\xi,\eta)\Ex\big[z^m_{\varepsilon,L,\eta}(s)\big]\d s
      \\
     &\quad +
       {{}\tfrac{\gamma^m}{2\varepsilon}} \int^{{{}T}}_0\Ex\big[\varepsilon\Gamma^m\big(z_{\varepsilon, L,\xi}(s)\big)z^m_{\varepsilon,L,\xi}(s)\big]\d s+\Ex\big[\sup_{0\le t\le T}\big| M^{G_L,\varepsilon}(t)\big|\big]
       \\
      &\le
       \Ex\big[z^m_\xi(0)\big]+
       \int^T_0 \sum_{\eta\in G_L}a(\xi,\eta)\Ex\big[z^m_{\varepsilon,L,\eta}(s)\big]\d s
      \\
     &\quad +
       \overline{\gamma}(\tfrac{1}{\varepsilon}\wedge \overline{K})\int^{{{}T}}_0\Ex\big[z^m_{\varepsilon,L,\xi}(s)\big]\d s+\Ex\big[\sup_{0\le t\le T}\big| M^{G_L,\varepsilon}(t)\big|\big].
    \end{aligned}
    \end{equation}

By a Burkholder-Davis-Gundy inequality and Cauchy-Schwartz's inequality, there is a $C<\infty$ such that
\begin{equation}
\begin{aligned}
\label{e:Doob}
   \Ex\big[\sup_{t \in [0,T]}\big|M^{G_L,\varepsilon}(t)\big|\big]
  &\le
     C\cdot\Ex\big[\langle M^{G_L,\varepsilon}({\boldsymbol{\cdot}})\rangle^{\frac{1}{2}}_{T}\big]
   \le
     C\cdot\big(\Ex\big[\langle M^{G_L,\varepsilon}({\boldsymbol{\cdot}})\rangle_{T}\big]\big)^{\frac{1}{2}}
     \\
   &\le
     C\cdot\big(1+\Ex\big[\langle M^{G_L,\varepsilon}({\boldsymbol{\cdot}})\rangle_{T}\big]\big)
     \\
   &\le
   \big(\overline{\gamma}+\varepsilon C\big)\int^T_0\Ex\big[\bar{z}_{\varepsilon,L,\xi}(s)\big]\mathrm{d}s
   +\varepsilon\overline{\gamma}\overline{\lambda}\int^T_0\Ex\big[\big(\bar{z}_{\varepsilon,L,\xi}(s)\big)^2\big]\mathrm{d}s\big)
\end{aligned}
\end{equation}
since
\begin{equation}
\label{e:qua}
\begin{aligned}
   &\langle M^{G_L,\varepsilon}({\boldsymbol{\cdot}})\rangle_{t}
   \\
   &=
    \tfrac{\gamma^m}{\varepsilon}\int^t_0\varepsilon^2\big(\tfrac{1}{2}+\varepsilon K^m\big)\tfrac{z^m_{\varepsilon,L,\xi}(s)}{\varepsilon}\mathrm{d}s
    +\tfrac{\gamma^m}{\varepsilon}\int^t_0\varepsilon^2\big(\tfrac{1}{2}+\varepsilon\sum_{n=1}^M\lambda_{m,n}z^m_{\varepsilon,L,\xi}(s)\big)
    \tfrac{z^m_{\varepsilon,L,\xi}(s)}{\varepsilon}\mathrm{d}s
    \\
    &=\gamma^m\int^t_0z^m_{\varepsilon,L,\xi}(s)\mathrm{d}s+\varepsilon\gamma^m\int^t_0z^m_{\varepsilon,L,\xi}(s)\big(K^m+\sum_{n=1}^M\lambda_{m,n}
    z^n_{\varepsilon,L,\xi}(s)\big)\mathrm{d}s
\end{aligned}
\end{equation}

Combining (\ref{Gr3}), (\ref{e:moment4}), (\ref{e:005ii}) and (\ref{e:Doob}) we can find $C(T)<\infty$ such that
    \begin{equation}
    \label{e:009}
       \sum_{\xi\in G}\rho(\xi) \Ex\big[\sup_{0\le t\le T}\bar{z}_{ \varepsilon,L,\xi}(t)\big]
       \le
      C(T)\sum_{\xi\in G}\rho(\xi) \Ex\big[\big(\bar{z}_\xi(0)\big)^2\big].
    \end{equation}

This completes the proof.
\end{proof}

\noindent
{\em Step~2 (Uniform convergence of generators) }
The next step is to show the following:
\begin{lemma}[Convergence of generators] {{}Let $f\in{\mathcal D}(\Omega^G_Z)$,
and denote by $f_{|L}$ its restriction to $(\R_+^M)^{G_L}$.
Then
\begin{equation}\label{e:appdis1}
   \lim_{L\to 0}\sup_{z\in(\R_+^M)^G}\big|\Omega_Z^{G_L} f_{|L}(z)-\Omega^G_Z f(z)\big|=0.
\end{equation}}
\label{L:dis1}
\end{lemma}

\begin{proof} Fix $f\in{\mathcal D}(\Omega^G_Z)$. Then
{{}
\begin{equation}\label{e:appdis2}
\begin{aligned}
  &\sup_{z\in(\R_+^M)^G}\big|\Omega_Z^{G_L} f(z)-\Omega^G_Z f(z)\big|
  \\
 &=
    \sum_{(\xi,\eta)\in G^2\setminus (G_L)^2}\sup_{z\in(\R_+^M)^G}\sum_{m=1}^Mz^m_\xi a(\xi,\eta)\big\{f\big(z+e(m,\eta)-e(m,\xi)\big)-f(z)\big\}
  \\
 &
   +\sum_{\xi\in G\setminus G_L}\sup_{z\in(\R_+^M)^G}\sum_{m=1}^M\gamma^m\big(\tfrac{1}{2}+K^m\big)z^m_\xi\Big\{f\big(z+e(m,\xi)\big)-f(z)\Big\}
  \\
 &\quad
   +\sum_{\xi\in G\setminus G_L}\sup_{z\in(\R_+^M)^G}\sum_{m=1}^M\gamma^m\big(\tfrac{1}{2}+\sum_{n=1}^M\lambda_{m,n}z^n_\xi\big)z^m_\xi
   \Big\{f\big(z-e(m,\xi)\big)-f(z)\Big\}
 &\Ltoo
   0,
\end{aligned}
\end{equation}
where we used that $\sup_{z\in(\R_+^M)^G}\Omega^G_Zf(z)<\infty$.}
\end{proof}\smallskip

\noindent
{\em Step~3 (Compact containment) }
The final step in establishing convergence to a solution of the martingale problem is the compact
containment condition. First we identify the compact sets.

\begin{lemma}[Compact sets in ${\mathcal E}$]
\label{L:Ecompact}
Let $A$ be a subset of ${\mathcal E}$. The set $A$ is compact in ${\mathcal E}$ equipped with the product topology if
\be{e:l1bou}
   \sup_{x \in A}\sum_{\xi \in G}\bar{x}_{\xi}\rho(\xi)=c<\infty.
\ee
\end{lemma}

\begin{remark}[Compact sets in {{}${\mathcal E}^{\mathrm{par},G}$}]\rm
 Since {{}${\mathcal E}^{\mathrm{par},G}$} is a closed subset of {{}${\mathcal E}^G$} the same statement holds for {{}${\mathcal E}^{\mathrm{par},G}$}. \label{Rem:01}
\hfill$\qed$
\end{remark}

\begin{proof}
Let $x^{(n)}$ be a sequence in {{}$A \subset {\mathcal E}^G$}, and let $\varepsilon>0.$ We have for each $m=1,...,M$, and $\xi\in G$,  $(x_{\xi}^m)^{(n)} \leq \frac{c}{\rho(\xi)}< \infty$ for all $n$ by (i). Therefore, for each choice of $m=1,...,M$, and $\xi\in G$, there exists a subsequence $n_i^{\xi,m}$ such that $ (x_{\xi}^m)^{(n_i^{\xi,m})}$ converges to some $\tilde{x}_{\xi}^m$ as $i \rightarrow \infty$. In fact, by a diagonalization argument we can find a common subsequence $n_i$ such that for all $m=1,...,M$, and $\xi\in G$,
\begin{equation}
\label{e:001}
 (x_{\xi}^m)^{(n_i)} \tio \tilde{x}_{\xi}^m
\end{equation}

By Fatou's lemma, $\sum_{\xi\in G}\rho(\xi)\sum_{m=1}^M\tilde{x}^m_\xi<\infty$ if (\ref{e:l1bou}) holds.
%
We have therefore have constructed a subsequence convergent in {{}${\mathcal E}^G$}, and
hence shown that $A$ is compact.
\end{proof}

\begin{lemma}[Compact containment]
For all $\varepsilon>0$ and $T>0$ there exists a compact set $A_{\varepsilon,T}\subset{\mathcal E}$ such that
\begin{equation}\label{e:dis2}
   \inf_{L\in\N}\Pr\big(X_L(t)\in A_{\varepsilon,T},\;\forall\,0\le t\le T\big)
 \ge
   1-\varepsilon.
\end{equation}
\label{L:dis2}
\end{lemma}\sm

\begin{proof}
    Fix $\varepsilon>0$ and $T>0$. Set
    \begin{equation}\label{kespT}
    \begin{aligned}
       K_{\varepsilon,T}
     :=
       \tfrac{1}{\varepsilon}\cdot C(T)\sum_{\xi\in G}\rho(\xi) \Ex\big[\bar{x}_\xi(0)\big]
    \end{aligned}
    \end{equation}
    with $C(T)$ as in (\ref{ee:moment4-1}), and put
    \begin{equation}\label{kespT2}
    \begin{aligned}
       A_{\varepsilon,T}
     :=
       \big\{z\in{\mathcal E}^{\mathrm{par},G}:\,\sum_{\xi\in G}\rho(\xi)\bar{z}_\xi\le K_{\varepsilon,T}\big\}.
    \end{aligned}
    \end{equation}

    Then $A_{\varepsilon,T}$ is compact by Lemma~\ref{L:Ecompact}, and for all $L\in\N$,
    \begin{equation}\label{e:003}
    \begin{aligned}
       \Pr\big(\sup_{0\le t\le T}\sum_{\xi\in G}\rho(\xi)\bar{z}_{L,\xi}(t)> K_{\varepsilon,T}\big)
     &\le
       \tfrac{1} {K_{\varepsilon,T}}\Ex\big[\sup_{0\le t\le T}\sum_{\xi\in G}\rho(\xi)\bar{z}_{L,\xi}(t)\big]
          =
       \varepsilon,
    \end{aligned}
    \end{equation}
    by Lemma~\ref{L:moments}.
\end{proof}

We conclude this subsection with the

\noindent
\begin{proof}{\bf of Theorem~\ref{T:parti}.} 
(i) Uniform convergence of the generator as given in~\ref{e:appdis1} together with the compact containment condition (\ref{e:dis2})
imply that the family $\{Z^{G_L,G};\,L\in\mathbb{N}\}$ is relatively compact by  \cite[Remark 4.5.2]{EthierKurtz1986}. \smallskip

(ii) Moreover, any limit point $Z^G$ satisfies the $(\Omega^G_{Z}, {\mathcal D}(\Omega^G_{Z}), z(0))$-martingale problem by \cite[Lemma 4.5.1]{EthierKurtz1986}. It also satisfies (\ref{e:poisson}).
This establishes {\em existence}. \smallskip

(iii) Recall the uniform first moment bound stated in (\ref{ee:moment4-1}).
We claim that this implies the following for any limit point $Z^G$, i.e.,
\begin{equation}
\label{e:009-a}
   \sum_{\xi\in G}\rho(\xi)\Ex\big[\sup_{t\in[0,T]}\bar{z}_\xi(t)\big]
   \le
   C(T)\sum_{\xi\in G}\rho(\xi)\Ex\big[\sup_{t\in[0,T]}\bar{z}_\xi(0)\big].
\end{equation}
Indeed, applying the Skorohod representation theorem, we can define $\{Z^{G_L};\,L\in\mathbb{N}\}$ and $Z^G$ on one and the same probability space
such that $Z^{G_L}\to Z^G$ in Skorohod topology almost surely, as $L\to\infty$. Thus also for each $T>0$, $\sup_{t\in[0,T]}z^m_{L,\xi}\to \sup_{t\in[0,T]}z^m_{\xi}$ almost surely, $L\to\infty$. We therefore have by Fatou's lemma, for all $\xi\in G$ and $m=1,...,M$, and $(L_k)\uparrow \infty$,
\be{Fat1}
\begin{aligned}
   \sum_{m=1}^M \sum_{\xi \in G} \rho(\xi) \Ex\big[\sup_{0 \leq t \leq T} z_{\xi}^m(t)\big]
  &\leq
   \liminf_{k \rightarrow \infty}\sum_{m=1}^M\sum_{\xi\in G}\rho(\xi)\Ex\big[\sup_{0 \leq t \leq T} z_{L_k,\xi}^m(t)\big]
   \\
   &\leq
   C(T) \sum_{\xi\in G}\rho(\xi) \Ex\big[\bar{z}_\xi(0)\big]
\end{aligned}
\end{equation}
by (\ref{e:009}).
\end{proof}

\subsection{Proof of Theorem~\ref{T0} and Theorem~\ref{Tdiff}}
\label{subsection:Tdiff}
We begin by proving Theorem~\ref{Tdiff}, which then will give the existence in
Theorem~\ref{T0}.
Here we proceed similarly as in the proof of Theorem \ref{T:parti}.
First we find a solution to the $(\Omega_X,X(0))$ martingale problem as the diffusion limit of the properly rescaled particle system. Here we use again Lemma 4.5.1 and Remark 4.5.2 of \cite{EthierKurtz1986}
to first show uniform convergence of the generators and then establish a compact containment condition.

\begin{lemma}[Convergence of generators]\label{L4}
Let $f\in{\mathcal D}(\Omega^G_X)$. Then
\begin{equation}\label{e:appdis1a}
   \lim_{\ve\to 0}\sup_{x\in\ve\CE^G}\big|\Omega^G_{Z^\varepsilon} f(x)-\Omega^G_{X} f(x)\big|=0.
\end{equation}
\label{L:dis1a}
\end{lemma}

\begin{proof}
{{} Fix $f\in{\mathcal D}(\Omega^G_X)$.
~
By the Taylor expansion, for all $\xi,\eta\in G$, $m=1,...,M$,
\begin{equation}
\label{tayloriii}
\begin{aligned}
   f\big(x\pm\varepsilon^{m}_{\xi}\big)
   &=
   {f(x)}\pm\varepsilon\tfrac{\partial }{\partial x^m_\xi}f(x)+\tfrac{1}{2}\varepsilon^2\tfrac{\partial^2 }{\partial (x^m_\xi)^2}f(x)+R^1(\pm\varepsilon,\xi,m),
\end{aligned}
\end{equation}
where
{
\begin{equation}
\label{taylorii}
   R^1(\pm\varepsilon,\xi,m)= \varepsilon^3\tfrac{\partial^3}{\partial (x^m_\xi)^3}f\big(x'(\varepsilon,\xi)\big)
\end{equation}
for some $x'(\varepsilon,\xi) \in [x-\varepsilon^{m}_{\xi},x+\varepsilon^{m}_{\xi}]$},
and
\begin{equation}
\label{taylori}
\begin{aligned}
   f\big(x+\varepsilon^{m}_{\eta}-\varepsilon^{m}_{\xi}\big)
   &=
   {f(x)+}\varepsilon\big(\tfrac{\partial }{\partial x^m_\eta}-\tfrac{\partial }{\partial x^m_\xi}\big)f(x)+R^2(\varepsilon,\eta,\xi,m),
\end{aligned}
\end{equation}
where
{
\begin{equation}
\label{taylorii}
   R^2(\varepsilon,\eta,\xi,m)
    =\tfrac{\varepsilon^2}{2}\big\{\tfrac{\partial^2}{\partial (x^m_\eta)^2}f\big(x''(\varepsilon,\xi,\eta)\big)+\tfrac{\partial^2}{\partial (x^m_\xi)^2}f\big(\tilde{x}''(\varepsilon,\xi,\eta)\big)\big\},
\end{equation}
where $x''(\varepsilon,\xi,\eta)\in[x+\varepsilon^{m}_{\eta}-\varepsilon^{m}_{\xi},x+\varepsilon^{m}_{\eta}]$ and
$\tilde{x}''(\varepsilon,\xi,\eta)\in[x,x+\varepsilon^{m}_{\eta}]$.}

{Notice that if $\Omega^G_Zf$ and $\Omega^G_Xf$ are bounded, then $R(\varepsilon,f;\boldsymbol{\cdot})$ is also bounded, where
\begin{equation}
\label{e:Resp}
\begin{aligned}
   &R(\varepsilon,f;x)
   \\
   &:=
   \Omega^G_Zf(x)-\Omega^G_Xf(x)
   \\
   &=
   \sum_{m=1}^M\sum_{\xi\in G} \tfrac{x_\xi^m}{\ve} \sum_{\eta\in G}a(\xi,\eta) R^2(\varepsilon,\eta,\xi,m)
        +\sum_{m=1}^M\tfrac{\gamma^m}{\varepsilon}\sum_{\xi\in G} \tfrac{x_\xi^m}{\varepsilon}\big(\tfrac{1}{2}+\varepsilon K^m\big)R^1(\varepsilon,\xi,m)
        \\
        &+\sum_{m=1}^M\tfrac{\gamma^m}{\varepsilon}
        \sum_{\xi\in G} \tfrac{x_\xi^m}{\varepsilon}\big(\tfrac{1}{2}+\varepsilon\sum_{n=1}^M\lambda_{m,n}x^n_\xi\big)R^1(-\varepsilon,\xi,m).
        \end{aligned}
        \end{equation}

Moreover, by (\ref{taylori}) and (\ref{taylorii}),
\begin{equation}
\label{e:008}
   \lim_{\varepsilon\to 0}\sup_{x\in{\mathcal E}^G}\big|R(\varepsilon,f,x)\big|=0,
\end{equation}
which proves the statement.
}

}
\end{proof}

The next lemma shows that the processes with
mass-$\varepsilon$-particles do not leave compact sets in the special case when $\lambda\equiv 0$
(state-independent supercritical branching),
i.e., when they solve the $(\Omega_{Z^\varepsilon,0},{\mathcal L}_b)$ martingale problem
corresponding to the operator
\be{gen3a}
\begin{aligned}
    {{}\Omega^G_{Z^{\varepsilon,0}} f(y)}
 &=
      \sum_{m=1}^M\sum_{\xi\in G} \tfrac{z_{\varepsilon,\xi}^m}{\ve} \big\{\sum_{\eta\in G}
    a(\xi,\eta) \big(f(z+{\ve}e(m,\eta)-{\ve}e(m,\xi))-f(z)\big)
    \\
    & \phantom{AAAAAAAAA} + {{}\gamma^mK^m}
    \big(f(z+\vee(m,\xi))-f(z)\big)
    \\
    & \phantom{AAAAAAAAA} +  {{}\tfrac{\gamma^m}{2\varepsilon}
    \big(f(z+\vee(m,\xi))+f(z-\vee(m,\xi))-2f(z)\big)\big\}.}
\end{aligned}
\ee
{{}applied to functions in
\begin{equation}
\label{ag50}
   {\mathcal D}(\Omega^G_{Z^{\varepsilon,0}}):= \big\{f\in B({\mathcal E}^G),\Omega^G_{Z^{\varepsilon,0}} f\in B({\mathcal E}^G)\big\},
\ee
}

\begin{lemma}[Compact containment] Let $x(0)$ be a random element in ${\mathcal E}^{\mathrm{par},G}$ such that $\sum_{\xi\in
    G}\rho(\xi)\Ex[\bar{x}_{\xi}(0)]<\infty$.
    Put for all $\varepsilon>0$, $\xi\in G$ and $m\in\{1,2,...,M\}$,
\be{e:exp0}
   z_{\varepsilon,\xi}^m(0):=\varepsilon\big\lfloor \tfrac{1}{\varepsilon}x_{\xi}^m(0)\big\rfloor.
\ee
Let $Z_{\varepsilon,0}:=(z^m_{\ve,0,\xi})_{m=1,...,M,\xi\in G}$ be a solution of the
$(\Omega^G_{Z^{\ve}},{\mathcal D}(\Omega^{G}_{Z^{\varepsilon,0}}),(z_{\varepsilon,\xi}^m(0))_{m=1,...,M,\xi\in G})$ martingale problem.
\label{L3}
    For all $\delta>0$ and $T>0$ there exists a compact set $A_{\delta,T}\subset{\mathcal E}^{\mathrm{par},G}$ such that
    \begin{equation}\label{e:dis2a}
       \inf_{\ve>0}\Pr\big\{Z_{\ve,0}(t)\in A_{\delta,T}\;\forall\,0\le t\le T\big\}
     \ge
       1-\delta.
    \end{equation}
\end{lemma}



\begin{proof}
By Lemma~\ref{L:Ecompact},
it  suffices to show that
  \begin{equation}\label{eq:dis2e}
    \begin{aligned}
      \sup_{\varepsilon>0}\Pr\Big(\sup_{0\leq t\leq T}
      z^m_{\varepsilon,0,\xi}(t) \geq L \Big) \xrightarrow{L\to\infty}
      0.
    \end{aligned}
  \end{equation}

  Indeed, assume (\ref{eq:dis2e}) and take $\delta>0$. We may also take a
  enumeration $(\xi_i)_i$ of $G$ and numbers $(L_i)_i$ with
  \begin{equation}
  \label{e:004}
     \sup_{\varepsilon>0 }\Pr\big\{\exists \, 0\le t\le T:\, x^m_{\varepsilon,\xi_i}(t) \geq L_i \big\}
     \leq
     \tfrac 1L\delta  2^{-i},
\end{equation}
 for all $i\in\N$ and $m\in\{1,2,...,M\}$.
  The set
  $A_{\delta,T}:=\otimes_{m=1}^M\otimes_{i=1}^\infty [0,L_i]$ is
  compact in the product topology. Moreover,
\be{e:moreover}
   \sup_{\varepsilon>0 }\Pr\big\{\exists 0\le t\le T:\, Z_{\varepsilon,0}(t)\notin A_{\delta,T}\big\} \leq
   \sum_{m=1}^M\sum_{i=1}^\infty \sup_{\varepsilon>0 } \Pr\big\{\exists 0\le t\le T:\, z^m_{\varepsilon,0,\xi_i}(t)
   > L_i \big\} \leq \delta.
\ee

{{}Since for each $W\in\mathbb{N}$,
  \begin{equation}\label{e:007b}
    \begin{aligned}
      M^m_{\varepsilon,0,\xi,W}:= \Big( W\wedge z^m_{\varepsilon,0,\xi}(t) -
      \int_0^t \Omega^G_{Z^{\varepsilon,0}} W\wedge z^m_{\varepsilon,0,\xi}(s)
      \mathrm{d}s\Big)_{t\geq 0}
    \end{aligned}
  \end{equation}
  is a martingale,
  \begin{equation}\label{eq:ppp134b}
    \begin{aligned}
      \sup_{\varepsilon>0}\Pr\big\{\sup_{0\leq t\leq T}
      W\wedge z^m_{\varepsilon,0,\xi}(t) \geq 2L \big\}
     &\leq
      \sup_{n\in\N}\Pr\big\{\sup_{0\leq t\leq T}
      M_{\varepsilon,0,\xi,W}^{m}(t) \geq L \big\}
      \\
     &+
       \sup_{\varepsilon>0}\Pr\big\{\sup_{0\leq t\leq T} \int_0^t
      (\Omega^G_{Z^{\varepsilon,0}} W\wedge z^m_{\varepsilon,0,\xi}(s))^+ \mathrm{d}s \geq L
      \big\}.
    \end{aligned}
  \end{equation}

We shall show separately for both terms on the right hand side to converge to
  0 as $L\to\infty$ uniformly in $W\in\mathbb{N}$. First for all $\varepsilon>0$, $W\in\mathbb{N}$
  \begin{equation}
    \label{eq:ppp135b}
    \begin{aligned}
      \Pr\big\{\sup_{0\leq t\leq T} & \int_0^t (\Omega_{Z^{\varepsilon,0}}
      z_{\varepsilon,0,\xi}^m(s))^+\mathrm{d}s \geq L \big\}
    \leq
      \tfrac 1L \Ex
      \big[ \int_0^T (\Omega^G_{Z^{\varepsilon,0}}
      W\wedge z_{\varepsilon,0,\xi}^m(s))^+\mathrm{d}s \big]
    \\
  & \leq
      \tfrac 1L \Ex\big[
      \int_0^T \big\{\big(\sum_{\eta\in G} a(\eta,\xi)
      z_{\varepsilon,0,\eta}(s)\big) + \overline{\gamma}\overline{K} z_{\varepsilon,0,\xi}(s) \mathrm{d}s\big\} \big]
    \\
  & \leq
      \tfrac {1}{L\rho(\xi)} \Ex\big[ \int_0^T\big(
      \sum_{\eta,\eta'\in G}
      a(\eta,\eta')\rho(\eta') z_{\varepsilon,0,\eta}(s)\big)+ \overline{\gamma}\overline{K} z_{\varepsilon,0,\xi}(s) \big) \mathrm{d}s
      \big]
    \\
  & \leq
      \tfrac{R+\overline{\gamma}\overline{K}}{L\rho(\xi)} \int_0^T \sum_{\eta\in G}
       \rho(\eta) \Ex[ z_{\varepsilon,0,\eta}(s)] \mathrm{d}s
    \\
  &\le
      \tfrac{(R+\overline{\gamma}\overline{K})}{L\rho(\xi)}e^{(C+1)T}
      \sum_{\eta\in G}\rho(\eta)\Ex\big[z_{\eta}^m(0)\big],
    \end{aligned}
  \end{equation}
  for $C$ as defined in (\ref{eq.gamma}) by (\ref{e:moment4}). Second, by the maximal inequality for
  martingales,
  \begin{equation}
    \label{eq:ppp136b}
    \begin{aligned}
      \Pr\big\{\sup_{0\leq t\leq T} & M^m_{\varepsilon,0,\xi,W}(t) \geq L \big\}
    \leq
      \tfrac 1L \Ex \big[ (M_{\varepsilon,0,\xi,W}^m(T))^+ \big]
      \\
    &\leq
      \tfrac 1L \Ex_0\big[z_{\varepsilon,0,\xi}^m(T)\big] + \tfrac 1L \Ex \big[
      \int_0^T (\Omega_{Z^{\varepsilon,0}} z_{\varepsilon,0,\xi}^m(s))^+\mathrm{d}s
      \big]
    \end{aligned}
\end{equation}
Here we have used that for all $a,b\in\R$, $(a-b)^+\le a^+-b^-$
and that in the special case where $\lambda\equiv 0$ (and only in this special case!), $(\Omega_{Z^{\varepsilon,0}} W\wedge z_{\xi}^m)^-\le (\Omega_{Z^{\varepsilon,0}} W\wedge z_{\xi}^m)^+$.}

  Combining \eqref{eq:ppp135b} and \eqref{eq:ppp136b} shows \eqref{eq:dis2e}.
  \end{proof}

\begin{proof}{\bf of Theorem~\ref{Tdiff}}\,
By explicit construction we can couple two finite {{}geographic space} particle systems {{}$Z^G$} and {{}$Z^{G,0}$} corresponding
respectively to the generator {{}$(\Omega^{G_L}_Z,{\mathcal D}(\Omega^{G_L}_Z))$ defined in (\ref{ag5}) and (\ref{pregenerator}) as well as
to the generator $(\Omega^{G_L}_{Z^0},{\mathcal D}(\Omega^{G_L}_{Z^0}))$ defined in (\ref{gen3a}) and (\ref{ag50})}
such that $z^m_{L,\xi}(0) = z^m_{L,0,\xi}(0)$ and $z^m_{L,\xi}(t) \leq  z^m_{L,0, \xi}(t)$ for all $t\geq 0$.
By Theorem~\ref{T:parti}(i)  
there exists limit points of such systems as $L\to\infty$ in the product topology. Hence there exist solutions to the corresponding
infinite {{}geographic space} particle systems that obey the same ordering.

Thus for any particle solution $Z^{\ve}$ to the $(\Omega^G_{Z^{\ve}},{\mathcal D}_{Z^{\ve}}(\Omega^G_{Z^{\ve}}),x(0))$-martingale problem constructed in Theorem~\ref{T:parti}
there exists a particle solution $Z_0^{\ve}$ of the  $(\Omega^G_{Z^{\ve,0}},{\mathcal D}(\Omega^G_{Z^{\ve,0}}), x(0))$-martingale problem such that again $z^m_{\ve, \xi}(t) \leq  z^m_{\ve,0, \xi}(t)$ for all $t\geq 0.$ Since by Lemma~\ref{L:dis2} the compact containment containment condition (\ref{e:dis2a}) holds for the family {{}$\{Z^{\ve,0},\ve>0\}$} we can conclude the same for the family {{}$\{Z^{\ve},\ve>0\}$}. Combining this with
Lemma \ref{L4} we can again apply Lemma 4.5.1 and Remark 4.5.2 of \cite{EthierKurtz1986} to complete
 the proof of Theorem \ref{Tdiff}.
 ~~
\end{proof}\sm



We first give the proof of Theorem  \ref{T0}.
{\em Existence} of a weak solution follows from the diffusion approximation stated in Theorem~\ref{Tdiff} together with Lemma~4.5.1 in \cite{EthierKurtz1986}.

\subsection{Moment estimation and comparison results}
\label{ss.momentcalc}
{{}Let $G$ be a countable Abelian group, and $X^G=\big\{(x^m_{\xi}(t))_{t\ge 0};\,\xi\in G,m=1,...,M\big\}$
a solution of (\ref{P:eq.007}). Recall from (\ref{Z001}) the total mass $\bar{x}_\xi$ in $\xi\in G$.
In this subsection we  state and prove some moment estimates for the total mass process.
The following implies Proposition~\ref{propn:l^pbound}.}

\begin{proposition}[Moment estimates]
Assume that the initial condition $X(0)$ has translation
    invariant total mass with $\Ex[(\bar x_\xi (0))^n] < \infty$ for all $\xi\in G$, $n\in\mathbb{N}$. Then for all $T \geq 0,$
    \be{GA1}
    \sup_{\xi \in G} \Ex \big[\sup_{0 \leq t \leq T} \big(\bar x_\xi(t)\big)^n\big] < \infty.
    \ee
    Furthermore there exists a $\delta>0$ depending on the parameters of the
    dynamics such that if $\Ex [\exp (\lambda \bar x_\xi (0))] < \infty$
    for $\lambda < \delta$  then for all $t \geq 0$:
    \be{GA2}
    \Ex \big[\exp \big(\lambda \bar x_\xi (t)\big)\big] < \infty
    \ee
    for all $\lambda < \delta$.
\label{PMom}
\end{proposition}

\begin{remark}
Note that this means that all mixed moments of the process
are measure-determining provided some exponential moment exists initially.
This implies for our purposes that we can proceed via moment calculations
since we can by truncation always achieve approximating processes for which
the assumptions are satisfied.
\hfill$\qed$
\label{Rem:02}
\end{remark}

\begin{proof}{\bf of Proposition \ref{PMom}}\;
We choose $(\theta^m,c^m)_{m=1,\dots,M}$ in such a way that
\be{eq.suchthat}
  \theta^m-c^m a_m> K^m a_m - \lambda_{m,m} a_m^2,
\ee
for any $(a_{m})_{m=1,\dots,M}\in \R_{+}^M$ and all $m\in\{1,...,M\}$.
    Hence
    \be{S:eq.0001}
        \theta^m-c^m a_m> a_m \left(K^m-\sum_{n=1}^M \lambda_{m,n} a_n\right),\quad m=1,\dots,M.
    \ee
    Then by It\^o's formula for each $m \in \{1,\dots, M\},$ and $\xi\in G$,
    \be{Z038}
    \begin{aligned}
        \big(x^m_\xi(t)\big)^n
        &= \big(x^m_\xi(0)\big)^n + \int_0^t \sum_{\eta\in G}{{}a}(\xi,\eta)  \big(x^m_{\eta}(s)-x^m_{\xi}(s)\big) n \big(x^m_{\xi}(s)\big)^{n-1} \d s\\
        &\;  + \gamma^m \int_0^t x^m_\xi(s)\bigg(K^m-\sum_{k=1}^M \lambda_{m,k}x^k_\xi(s)\bigg) n \big(x^m_\xi(s)\big)^{n-1}  \d s \\
        &\;  + \frac{\gamma^m}{2} n(n-1) \int_0^t \big(x^m_\xi(s)\big)^{n-1} \d s +   \int_0^t \sqrt{\gamma^m x^m_{\xi}(s)}\, n  \big(x^m_{\xi}(s)\big)^{n-1} \d w_\xi(s) \\
        &\le \big(x^m_\xi(0)\big)^n + \int_0^t \sum_{\eta\in G}{{}a}(\xi,\eta)  \big(x^m_{\eta}(s)-x^m_{\xi}(s)\big) n \big(x^m_{\xi}(s)\big)^{n-1} \d s\\
        &\;+ \gamma^m \int_0^t \big(\theta^m- c^m x^m_\xi(s)\big) n \big(x^m_\xi(s)\big)^{n-1}  \d s \\
        &\,+ \tfrac{\gamma^m}{2} n(n-1) \int_0^t \big(x^m_\xi(s)\big)^{n-1} \d s
     + \int_0^t \sqrt{\gamma^m x^m_{\xi}(s)}\, n  \big(x^m_{\xi}(s)\big)^{n-1} \d w_\xi(s).
    \end{aligned}
    \ee

    Due to translation invariance the distribution of $x^m_{\xi}(t)$ is identical to the distribution of $x^m_{\eta}(t)$ for any {{}$\xi,\eta \in G$, and $m=1,...,M$}. An application of H\"older's inequality implies therefore that
    $\Ex\big[x^m_\eta(t) \big(x^m_{\xi}(t)\big)^{n-1}\big]\leq \Ex\big[\big(x^m_{\xi}(t)\big)^{n}\big]$.
    Hence,
    \be{Z039}
    \begin{aligned}
        \Ex\big[\big(x^m_\xi(t)\big)^n\big]
        &\leq \Ex\big[\big(x^m_\xi(0)\big)^n\big] + \int_0^t \Ex\Big[ \gamma^m \big(\theta^m- c^m x^m_\xi(s)\big)  n \big(x^m_\xi(s)\big)^{n-1}\Big] \d s \\
        & + \frac{\gamma^m}{2} n(n-1) \int_0^t \Ex\big[ \big(x^m_\xi(s)\big)^{n-1}\big] \d s.
    \end{aligned}
    \ee

    Due to the positivity we obtain
    \be{Z040}
        \Ex\big[\big(x^m_\xi(t)\big)^n\big]  \leq \Ex\big[\big(x^m_\xi(0)\big)^n\big] + \gamma^m \big(n \theta^m + \binom{n}{2}\big)\int_{0}^{t}\Ex\big[\big(x^m_\xi(s)\big)^{n-1}\big] \d s.
    \ee

    Hence, again by translation invariance and the moment conditions at time $t=0,$ we obtain that for
    any $T>0$, and $n \in \N$,
    \be{indstepmoments}
        \sup_{\xi \in G} \sup_{0\leq t \leq T}  \Ex\left[\big(x^m_\xi(t)\big)^{n-1}\right]< \infty
        \quad \Rightarrow \quad \sup_{\xi \in G} \sup_{0\leq t \leq T}  \Ex\left[\big(x^m_\xi(t)\big)^{n}\right]< \infty
    \ee
    provided that $\Ex\left[(x^m_\xi(0))^{n}\right]< \infty.$
    But  since
    \be{Z041}
        \frac{\d}{\d t}\Ex\left[x^m_\xi(t)\right]
        = \gamma^m \Ex\left[x_\xi^m(t)\left(K^m-\sum_{k=1}^M \lambda_{m,k}x^k_\xi(t)\right)\right] < \gamma^m \left(\theta^m -  c^m \Ex[x^m_\xi(t)]\right),
    \ee
    we have $\Ex\big[x^m_\xi(t)\big]<u(t)$ where $u(t)$ is the solution to
    \be{Z042}
        \frac{\d}{\d t}u(t)
        = \gamma^m \big(\theta^m -  c^m u(t)\big),\quad u(0)=\Ex\big[x^m_\xi(0)\big],
    \ee
    which is
    \be{Z043}
        u(t) = \frac{\theta^m}{c^m} + \left(\Ex[x^m_\xi(0)] - \frac{\theta^m}{c^m} \right) e^{-\gamma^m c^m t}.
    \ee
    Hence
    \be{Z044}
        \Ex\big[x^m_\xi(t)\big]<\frac{\theta^m}{c^m} + \left(\Ex\big[x^m_\xi(0)\big] - \frac{\theta^m}{c^m} \right) e^{-\gamma^m c^m t}.
    \ee
    It now follows from (\ref{indstepmoments}) by induction on $n$ that
    \be{supE}
        \sup_{\xi \in G} \sup_{0\leq t \leq T}  \Ex\big[\big(x_\xi^m(t)\big)^{n}\big]< \infty
    \ee
    if $ \sup_{\xi \in G} \Ex[\bar{x}_\xi(0)]< \infty.$

    In order to improve (\ref{supE}) we observe that by
    Jensen's inequality, for some constant $c(n, T),$ and all
    $0 \leq t \leq T,$
    \be{Z045}
    \begin{aligned}
        \big(x^m_\xi(t)\big)^n
        &\leq c(n, T) \Bigg(  \big(x^m_\xi(0)\big)^n + \int_0^t \sum_{\eta\in G} a(\xi,\eta) \big|x^m_{\eta}(s)-x^m_{\xi}(s)\big|^n \d s\\
        & \phantom{AAAAAA} + \int_0^t \gamma^m \big|\theta^m- c^m x^m_\xi(s)\big|^n   \d s + \left|\int_0^t  \sqrt{\gamma^m x^m_{\xi}(s)} \, \d w_\xi(s) \right|^n \Bigg),
    \end{aligned}
    \ee
    By Burkholder's inequality, translation invariance and using the bound $|a-b|^n\leq (2a)^n + (2b)^n$ for $a,b\geq 0$ it now follows that
    \be{Z046}
    \begin{aligned}
        \sup_{\xi \in G}\Ex\bigg[  \sup_{0 \leq t \leq T} \big(x^m_\xi(t)\big)^n\bigg]
        &\leq c(n, T) \Bigg( \sup_{\xi \in G}  \Ex\big[ \big(x^m_\xi(0)\big)^n\big]  \\
        & \quad +\big(2^{n+1}+ (2 \gamma^m c_m)^n\big) \int_0^T \sup_{\xi \in G}\Ex\big[ \big(x^m_\xi(s)\big)^n\big]\,   \d s\\
        & \quad + (2 \gamma^m \theta^m)^n T+ \sup_{\xi \in G}\Ex\left[ \left(\int_0^T  \gamma^m x^m_{\xi}(s)\,  \d s \right)^{n} \right]^{\frac{1}{2}} \Bigg) .
    \end{aligned}
    \ee
    Combining this with (\ref{supE}) implies (\ref{GA1}).
    For the exponential moments use that an exponential moment exists if this is
    true for the non-spatial part.
     Notice that (\ref{Z040}) implies that in the non-spatial case,
      \be{Z0400}
        \Ex\big[\big(x(t)\big)^n\big]  \leq \Ex\big[\big(x(0)\big)^n\big] + \overline{\gamma}\big(2\overline{\theta}+1\big) \binom{n}{2}\int_{0}^{t}\Ex\big[\big(x(s)\big)^{n-1}\big] \d s.
    \ee
    Notice that if we replace $\leq$ by $=$, we obtain the moments of the Feller branching diffusion with branching rate $\overline{\gamma}(2\overline{\theta}+1)$, which is known to have exponential moments for suitably small exponents.
 \end{proof}

{{}We next focus on the exchangeable model only. }
\begin{proposition}\label{PMom-exch}
    In the exchangeable case we also have for all $n \in \N$ and $0<s\leq T<\infty$
    \be{S-exch1}
    \sup_{\xi \in G}\sup_{t\ge s} \Ex \big[\big(\bar x_\xi(t)\big)^n\big] < \infty,\quad \sup_{\xi \in G} \Ex \bigg[\sup_{s \leq t\leq T } \big(\bar x_\xi(t)\big)^n\bigg] < \infty,
    \ee
    and in particular if $\bar{X}$ is started from an initial condition $\bar{X}(0)$ bounded above by
    a translation invariant $\bar{X}^{inv}(0)$
    with $\Ex[(\bar x^{inv}_\xi (0))^n] < \infty$ 
    then (\ref{GA1}) holds and so does
    \be{S-exch2}
    \sup_{\xi \in G} \sup_{t\ge 0}\Ex \big[ \big(\bar x_\xi(t)\big)^n\big] < \infty.
    \ee
\end{proposition}

We will also need a comparison result for solutions to the total mass process in the exchangeable case
or equivalently the one type model with immigration. We will also apply the result to the $\alpha$ process in (\ref{e:thedual2}). We therefore state the comparison result in some generality.

\begin{proposition}\label{monotonicity}
Let $Y^{(i)}$ for $i=1,2$ be two stochastic processes with values in ${\mathcal E}$ such that
$y^{m (1)}_{\xi}(t) \leq y^{m (2)}_{\xi}(t)$ for all $\xi \in G, m=1, \dots, M$ and $t\ge 0$.
Let $f^{(i)}: {\cal E} \to {\cal E}$ be continuous functions for $i=1,2$ and assume that
for $x^{(1)}, x^{(2)} \in (\R_+)^{G\times \{1, \dots, M\}}$
\begin{eqnarray}
\label{compassump}
& &\sum_{\xi \in G}  \sum_{m=1}^M \Big( f^{m (1)}_{\xi}(x^{(1)}) - f^{m(2)}_{\xi}( x^{(2)} ) \Big)
1_{\big\{ x^{m (1)}_\xi-x^{m (2)}_\xi>0\big\} }\rho(\xi)\\
\nonumber
&\leq& c  \sum_{\xi \in G}  \sum_{m=1}^M
\big(x^{m (1)}_\xi-x^{m (2)}_\xi\big)1_{\big\{ x^{m (1)}_\xi-x^{m (2)}_\xi>0\big\} }\rho(\xi).
\end{eqnarray}
Suppose that  $x^{(i)}$ take values in ${\mathcal E}$ for $i=1,2$ and
are solutions to
\begin{equation}
\label{e:comp}
   \d x_{\xi}^{m (i)}(t)
 = f^{m (i)}\big(x^{(i)}(t)\big) \d t
   +\sqrt{\gamma^m x_{\xi}^{m (i)}(t)}\,\d w_{\xi}^m(t)+ y^{m (i)}_{\xi}(t)\,\d t
\end{equation}
with respect to the same family of independent Brownian
motions $\{w_{\xi}^m; \xi\in G, m=1, \dots, M\}$ and such that $x^{m (1)}_{\xi}(0) \leq x^{m (2)}_{\xi}(0)$ for all $\xi \in G.$ Then, we have that $x^{m (1)}_{\xi}(t) \leq x^{m (2)}_{\xi}(t)$ a.s.  for all $t \geq 0$ and $\xi \in G, m=1, \dots, M.$
\end{proposition}
\bi
\noi
{\bf Remark} \quad
Let $M=1.$ Condition (\ref{compassump}) is satisfied if for $i=1,2$
\begin{equation}
f^{(i)}(x)=  \sum_{\eta\in G}a(\xi,\eta)\big(x_\eta-x_{\xi}\big)+\gamma x_{\xi} \big(K-\lambda x_{\xi}\big)
\end{equation}
This is due to (\ref{Gr3}) and the fact that we are considering $x^{(i)} \in (\R_+)^G.$

\bi

\begin{proof}{\bf of Proposition \ref{PMom-exch}}\;
    From Theorem 2 of \cite{HW07} we know that there exists a translation invariant maximal process $\bar{X}^{(\infty)}=(\bar{x}_{\xi}^{(\infty)}(t))_{\xi \in G, t >0},$ also a solution to
    (\ref{P:eq.007}) for $M=1$ such that
    for all $t >0$ and $\xi \in G,$ $\bar{x}_{\xi}(t)$ is stochastically smaller than $\bar{x}_{\xi}^{(\infty)}(t).$ In order to prove the first part of
    (\ref{S-exch1}) it therefore suffices to consider the process $\bar{X}^{(\infty)}$
    which decreases stochastically as $t \rightarrow \infty$ and which satisfies $\Ex[\bar{x}_{\xi}^{(\infty)}(t)]<\infty$ for any $t>0, \xi \in G,$ again by Theorem 2 of
    \cite{HW07}. Due to the translation invariance this implies (\ref{S-exch1}) immediately for $n=1.$ For $n\geq 1,$ we proceed by induction. We calculate with It\^o's formula,
    \be{Z053}
    \begin{aligned}
        \d (\bar{x}_\xi^{(\infty)}(t))^n
        &= \sum_{\eta\in G}a(\xi,\eta) \big(\bar{x}_{\eta}^{(\infty)}(t)-\bar{x}_{\xi}^{(\infty)}(t)\big) n \big(\bar{x}_{\xi}^{(\infty)}(t)\big)^{n-1} \d t\\
        & \quad + \gamma n \big(\bar{x}_\xi^{(\infty)}(t)\big)^{n} \big( K- \lambda   \bar{x}_{\xi}^{(\infty)}(t)\big)\,  \d t +\frac{\gamma}{2} n(n-1) \big(\bar{x}_\xi^{(\infty)}(t)\big)^{(n-1)}   \d t\\
        & \quad + \sqrt{\gamma \bar{x}_{\xi}^{(\infty)}(t)}\, n  \big(\bar{x}_{\xi}^{(\infty)}(t)\big)^{n-1} \d w_\xi(t).
    \end{aligned}
    \ee
    Due to translation invariance  the distribution of $\bar{x}_{\xi}^{(\infty)}(t)$ is identical to the distribution of
    $\bar{x}_{\eta}^{(\infty)}(t)$ for any $\eta \in G.$ As before, an application of H\"older's inequality implies therefore that $\Ex\big[(\bar{x}_\eta^{(\infty)}(t) (\bar{x}_{\xi}^{(\infty)}(t))^{n-1}\big]\leq \Ex\big[(\bar{x}_{\xi}^{(\infty)}(t))^{n}\big].$
    Hence,
    \be{Z054}
    \begin{aligned}
        \d \Ex\big[(\bar{x}_\xi^{(\infty)}(t))^n\big]
        &\leq \gamma n K \Ex\big[(\bar{x}_\xi^{(\infty)}(t))^{n}\big]  \d t -\gamma n  \lambda \Ex\big[(\bar{x}_\xi^{(\infty)}(t))^{n+1}\big]   \d t\\
        & \quad  +\frac{\gamma}{2} n(n-1) \Ex\big[ (\bar{x}_\xi^{(\infty)}(t))^{(n-1)}\big]   \d t
    \end{aligned}
    \ee
    or, more precisely, for $0 \leq s \leq t<\infty,$
    \be{Z055}
    \begin{aligned}
        & \int_s^t \Ex\big[(\bar{x}_\xi^{(\infty)}(u))^{n+1}\big] \d u \\
        &\leq \frac{1}{\gamma n  \lambda} \left( \Ex\big[(\bar{x}_\xi^{(\infty)}(s))^n\big] - \Ex\big[(\bar{x}_\xi^{(\infty)}(t))^n\big]  \right) + \frac{K}{\lambda} \int_s^t \Ex\big[(\bar{x}_\xi^{(\infty)}(u))^{n}\big]  \d u \\
        & \quad +\frac{n-1}{2 \lambda} \int_s^t \Ex\big[ (\bar{x}_\xi^{(\infty)}(u))^{(n-1)}\big]   \d u
    \end{aligned}
    \ee
    Thus, if (\ref{S-exch1}) is true for $n$ and  $n-1$
    then the left hand side is finite.
    Using that $u \mapsto  (\bar{x}_\xi^{(\infty)}(u))^{n+1}$ is stochastically decreasing thus implying that
    $u \mapsto  \Ex[(\bar{x}_\xi^{(\infty)}(u))^{n+1}]$ is decreasing as well as the translation invariance of $X^{(\infty)}$ the first statement of (\ref{S-exch1}) now follows for $n+1.$ For the second statement of  (\ref{S-exch1})  we use that due to the positivity of the solutions and  H\"older's inequality there exists a constant $c=c(n,T)$ such that for
    $0 \leq s \leq T< \infty,$
    \be{Z056}
    \begin{aligned}
        & \Ex \Big[\sup_{s \leq t\leq T } (\bar x_\xi(t))^n\Big] \\
        &\leq c\Bigg( \Ex [(\bar x_\xi(s))^n] + 2^n  \sum_{\eta\in G} a(\xi,\eta) \int_s^T \Ex\big[(\bar{x}_\eta(t))^n\big] \d t \\
        & \quad + \big(2^n+(\gamma K)^n\big) \int_s^T \Ex\big[(\bar{x}_\xi(t))^n\big] \d t +   \gamma^{\frac{n}{2}} \bigg(\int_s^T \Ex[(\bar{x}_\xi(t))^{n}]\, \d t \bigg)^{\frac{1}{2}}\Bigg).
    \end{aligned}
    \ee
    where we have also used that by  Burkholder's inequality
    \be{Z057}
    \begin{aligned}
        \Ex\Bigg[\sup_{s \leq t\leq T } \bigg(\int_s^t \sqrt{\bar{x}^m_\xi(u)}\,  \d w_\xi(u) \bigg)^{n} \Bigg]
        &\leq  \Ex\Bigg[\sup_{s \leq t\leq T } \bigg(\int_s^t \sqrt{\bar{x}^m_\xi(u)}\, \d w_\xi(u)\bigg)^{2n} \Bigg]^{ \frac{1}{2} }\\
        &\leq  \Ex\Bigg[ \bigg(\int_s^T \bar{x}^m_\xi(t)\, \d t\bigg)^{n} \Bigg]^{ \frac{1}{2} }.
    \end{aligned}
    \ee
    This means that
    \be{Z058}
        \sup_{\xi \in G}\sup_{s \leq t \leq T} \Ex \big[\big(\bar x_\xi(t)\big)^{n}\big] < \infty \quad \text{ implies } \quad \sup_{\xi \in G} \Ex \bigg[\sup_{s \leq t\leq T } \big(\bar x_\xi(t)\big)^n\bigg] < \infty,
    \ee
    thus completing the proof of  (\ref{S-exch1}). In order to prove (\ref{S-exch2}) we first note that due to
    monotonicity in the initial condition (see Proposition \ref{monotonicity} and the following remark)
    it suffices to consider a process with translation invariant initial conditions.
    Thus (\ref{GA1}) follows from Proposition \ref{PMom}. Combining this fact with
    (\ref{S-exch1}) finishes the proof.
\end{proof}

\begin{proof}{\bf of Proposition \ref{monotonicity}}\:
 This result follows from an adaptation of fairly standard methods of Yamada and Watanabe
    \cite{YamadaWatanabe1971}, see also Shiga and Shimizu \cite{ShigaShimizu1980}.
    We want to show that  $X=X^{(1)} -X^{(2)}$ is not positive. Here, $X$ solves
    \begin{eqnarray}
    \nonumber
        \d x_{\xi}^m(t)
        = \Big( f^{m(1)}_{\xi}\big( x^{(1)}(t)\big) - f^{m (2)}_{\xi}\big( x^{(2)}(t)\big) \Big) \d t &+\bigg( \sqrt{\gamma^m x_{\xi}^{m (1)}(t)}-\sqrt{\gamma^m x_{\xi}^{m (2)}(t)} \bigg) \d w_{\xi}^m(t)\\
        \label{Z059}
         &+ \Big(y^{m (1)}_{\xi}(t) -y^{m (2)}_{\xi}(t) \Big) \d t
    \end{eqnarray}
    Let $g^{+}(x)=x \vee 0,\, x \in \R$ and let $g^{+, n}$ be an appropriate smoothing of $g^{+}$ as in \cite{YamadaWatanabe1971}. We can choose
    $g^{+, n}$ such that  $g^{+, n} \uparrow  g^{+}$ uniformly as $n \rightarrow \infty$ as well as
    $0\leq (g^{+, n})' \uparrow  1_{(0,\infty)}$ and
    $0 \leq (g^{+, n})''(x) \leq\frac{2}{n |x|}.$   We apply It\^o's formula to $x_{\xi}^m$ with the function $g^{+, n}$
    and consider the result stopped at time
    $ T_{N}= \inf\big\{t \geq 0 : \sum_{\xi \in G} \sum_{m=1}^M |x_{\xi}(t)| \rho(\xi) \geq N \big\}. $
    Using that  $g^{+,n}(x_{\xi}(0))= 0$ and $y^{m (1)}_{\xi}(t) -y^{m (2)}_{\xi}(t)\leq 0$ by assumption we get after    taking expections that
    \be{Z060}
    \begin{aligned}
        \Ex\Big[ g^{+,n}(x_{\xi}^m(t \wedge  T_{N})\Big]
        &\leq \Ex\Bigg[ \int_{0}^{t \wedge  T_{N}}  (g^{+,n})'(x_{\xi}^m(s))
        \Big( f^{m(1)}_{\xi}\big( x^{(1)}(s)\big) - f^{m (2)}_{\xi}\big( x^{(2)}(s) \big) \Big) \d s \Bigg]\\
        & \quad +\frac{1}{2} \Ex\Bigg[ \int_{0}^{t \wedge  T_{N}}  (g^{+,n})''(x_{\xi}^m(s))\Big( \sqrt{ \gamma^m x^{m (1)}_{\xi}(s)} -  \sqrt{ \gamma  x^{m (2)}_{\xi}(s) } \,\Big)^2 \d s \Bigg].
    \end{aligned}
    \ee
    Letting $n \rightarrow \infty,$ the quadratic variation term vanishes and we obtain due to Lebesgue's dominated convergence theorem that
    \be{Z061}
        \Ex\Big[ g^{+}(x_{\xi}^m(t \wedge  T_{N})\Big] \leq \Ex\Bigg[ \int_{0}^{t \wedge  T_{N}} 1_{\{x_{\xi}^m(s)>0\}}
        \Big( f^{m (1)}_{\xi}\big( x^{(1)}(s)\big) - f^{m (2)}_{\xi}\big( x^{(2)}(s) \big) \Big) \d s \Bigg].
    \ee
    We multiply the above by $\rho(\xi)$ and sum over $\xi$ and $m.$
    Using (\ref{compassump}) we arrive at
    \be{Z062}
        \Ex\Bigg[ \sum_{\xi \in G} \sum_{m=1}^M g^{+}\big(x_{\xi}^m(t \wedge  T_{N}\big) \rho(\xi) \Bigg] \leq c \int_{0}^{t }\Ex\Bigg[ \sum_{\xi \in G} \sum_{m=1}^M g^{+}\big(x_{\xi}^m(s\wedge  T_{N})\big) \rho(\xi) \Bigg] \d s
    \ee
    An application of Gronwall's lemma implies now that
    $$\Ex\Big[ \sum_{\xi \in G} \sum_{m=1}^M g^{+}\big(x_{\xi}^m(t \wedge  T_{N})\big) \rho(\xi) \Big]=0.$$
    Since $T_{N} \uparrow \infty$ a.s. as $N \rightarrow \infty$ and application of Fatou's lemma proves our result.
\end{proof}

\section{Proof of exponential duality in the exchangeable model}\label{P:exdual}

\begin{proof}{\bf of Lemma \ref{L.mart}}\;
For each $\kappa=( \kappa^{m}_{\xi})_{m=1,\dots,M, \xi \in G} \in (\N_0^M)^G$ with $\sum_{\xi \in G} \bar{\kappa}_{\xi}< \infty$ we can use an explicit construction with exponential random variables to see that there exists a unique Markov process $\kappa(t) \in D(\R_+,  (\N_0^M)^G)$ such that
$\kappa(0)=\kappa$ corresponding to the formal generator
\begin{eqnarray*}
   \Omega_{\kappa}f\big(\kappa)
 &:=&
   \sum_{m=1}^M\sum_{\xi,\eta\in G} \kappa^{m}_{\xi} a(\xi,\eta) \Big(f\big(\kappa+e(m,\eta)-e(m,\xi)\big)-f(\kappa)\Big)
  \\
 &&\qquad+
 \gamma\sum_{m=1}^M\sum_{\xi\in G}{\kappa^{m}_{\xi}\choose 2}
 \Big(f(\kappa-e(m,\xi))-f(\kappa)\Big).
\end{eqnarray*}
%
%
We also have that $\sum_{\xi \in G} \bar{\kappa}_{\xi}(t) \leq \sum_{\xi \in G} \bar{\kappa}_{\xi}< \infty.$
This follows immediately since by assumption the initial number of particles is finite and thus also the transition rates. Due to the nature of the transitions, the number of particles may only decrease over time.

Let $k_{\xi} \in \N^G$ such that $\sum_{\xi} k_{\xi}< \infty.$ Then, for any $\alpha(0) \in l^p(\rho)$ for $p\geq 4$
there exists a unique solution $\alpha$ with initial condition $\alpha(0)$ and such that $\alpha(t) \in l^p(\rho)$ for all $t \geq 0$ a.s. corresponding to the generator
\begin{eqnarray*}
   \Omega_{\alpha}^{k}f\big(\alpha)
 &:=&
   \sum_{\xi,\eta\in G}\bar{a}(\xi,\eta)\big(\alpha_\eta-\alpha_{\xi}\big)\frac{\partial}
   {\partial \alpha_{\xi}}f\big(\alpha,\kappa\big)\\
 & &\qquad+
   \gamma\sum_{\xi\in G}\alpha_{\xi} \Big(K-\frac12\alpha_{\xi}\Big)\frac{\partial}
   {\partial
\alpha_{\xi}}f\big(\alpha,\kappa\big)
+\gamma\lambda\sum_{\xi\in G}\alpha_{\xi}\frac{\partial^2}
   {\partial (\alpha_{\xi})^2}f\big(\alpha,\kappa\big) \\
 & &\qquad+
   \gamma\sum_{\xi\in G}\lambda k_{\xi}\frac{\partial}
   {\partial
\alpha_{\xi}}f\big(\alpha,\kappa\big),
\end{eqnarray*}
To see this we may apply Theorem 2.3 in \cite{Sturm03}, see also the proof of Theorem 5.1 in \cite{BlathEtheridgeMeredith2007} who apply this theorem in a setting similar to ours. The linear growth condition on the diffusion term needed in Theorem 2.3 in \cite{Sturm03} is immediate. Due to (\ref{Gr3}) and  $\sum_{\xi} k_{\xi}< \infty$ it is not hard to see that for $p \geq 1,$
\begin{eqnarray*}
\sum_{\xi \in G}  |\sum_{\eta\in G}\bar{a}(\xi,\eta)\big(\alpha_\eta-\alpha_{\xi}\big)+
   \gamma \alpha_{\xi} (K-\frac12\alpha_{\xi}) +   \gamma \lambda k_{\xi} |^p \rho(\xi)
&\leq& C  (1+ ||\alpha||_{2p, \rho}).
\end{eqnarray*}
In fact, we could have bounded the left hand side by $(1+ ||\alpha||_{p, \rho})$ if it were not for the quadratic term which only allows for the above bound. The arguments in the proof of Theorem 2.3 of \cite{Sturm03} then imply the existence of a process $\alpha$ to the system of stochastic differential equations corresponding to $  \Omega_{\alpha}^{k}$ for each initial condition $\alpha(0) \in  l^{2p}(\rho)$ with $p\geq 2.$ The solution is continuous in each component and satisfies
$\E( \sup_{0 \leq t \leq T}  ||\alpha||_{p, \rho}^{p}) < \infty$  for each $T \geq 0.$  The weak uniqueness of the process follows with a Girsanov argument as in the proof of Theorem \ref{T0}.
But in fact, by a slight modification of Proposition \ref{propn:l^pbound}  and its proof to accommodate the extra immigration term we can now conclude that $\alpha(t) \in l^{2p}(\rho)$ for all $t \geq 0$ a.s. and that $\E( \sup_{0 \leq t \leq T}  ||\alpha||_{2p, \rho}^{2p}) < \infty.$ This yields the claimed result.

It is therefore also immediate that for any sequence of times $0=t_0 \leq t_1\leq t_2 \leq ...$ and sequence $k^{(i)} \in \N^G$ with $\sup_i \sum_{\xi} k_\xi^{(i)} < \infty$  and $\alpha(0) \in l^p(\rho)$
 for $p \geq 4$ there exists a unique process $\alpha$ with initial condition $\alpha(0)$ such that $\alpha(t) \in l^p(\rho)$ for all $t \geq 0$ a.s. that is a solution of the system of stochastic differential equations corresponding to the generator $ \Omega_{\alpha}^{k^{(i)}}$ on the time interval $[t_{i-1},t_i].$
Each realization of an independent process $\kappa$ provides  such a sequence of (jump) times and states. Thus, for each $\alpha(0) \in l^p(\rho)$  for $p \geq 4$ and  $\kappa(0) \in (\N_0^M)^G$ with $\sum_{\xi \in G} \bar{\kappa}_{\xi}(0)< \infty$ we can define an $(\alpha, \kappa)$ process on a joint probability space that has $\Omega_{(\alpha, \kappa)}$ as its generator, and this process is unique in law. As before we have that the $\alpha$ process  is continuous in each component with $\alpha(t) \in l^{p}(\rho)$ for all $t \geq 0$ a.s. and $\E( \sup_{0 \leq t \leq T}  ||\alpha||_{p, \rho}^{p}) < \infty$ as well as that
$\kappa(t) \in D(\R_+,  (\N_0^M)^G)$ with $\sum_{\xi \in G} \bar{\kappa}_{\xi}(t) \leq \sum_{\xi \in G} \bar{\kappa}_{\xi}(0)< \infty.$\\
~
\end{proof}

\begin{proof}{\bf of Lemma \ref{ZL1}}\;
    Fix $t\ge 0$.
    Note that in the process $\kappa(t)$ particles only coalesce or migrate. Since $\kappa(0)$ has finite support,
    it follows immediately that $\sum_{\xi\in G}\bar\kappa_\xi(t) \bar{x}_\xi<\infty$ for all $t\ge 0$, a.s.

    For the process $\alpha(t)$ we prove the stronger statement
    \be{Z025}
        \Ex\Big[\sum_{\xi\in G}\alpha_\xi(t) \bar{x}_\xi\Big]<\infty \quad\forall x\in\CE,(\alpha(0),\kappa(0))\in\CE_f^{\rm dual} .
    \ee
    Note that this term can be compared with the corresponding term for a system of interacting supercritical Feller diffusions with immigration. Namely, by Proposition~\ref{monotonicity},
    \be{Z026}
        \Ex\Big[\sum_{\xi\in G}\alpha_\xi(t) \bar{x}_\xi\Big] \le \Ex\Big[\sum_{\xi\in G}y_\xi(t) \bar{x}_\xi\Big],
    \ee
    where $(y(t),\pi(t))_{0\le t\le T}$ is a Markov process with $(y(0),\pi(0))=(\alpha(0),\bar{\kappa}(0))$ and with the following generator,
    \be{Z027}
        \begin{aligned}
            \Omega_{y,\pi}f(y,\pi) :=
            & \sum_{\xi,\eta\in G} \pi_{\xi} a(\xi,\eta) \Big(f\big(y,\pi+\delta_{\eta}-\delta_{\xi}\big)-f(y,\pi)\Big) \\
            & {}+ \sum_{\xi,\eta\in G}\bar{a}(\xi,\eta)\big(y_\eta-y_{\xi}\big)\frac{\partial f}{\partial y_{\xi}}(y,\pi) + \gamma K\sum_{\xi\in G}y_{\xi} \frac{\partial f}{\partial y_{\xi}}(y,\pi) \\
            & {}+ \gamma \lambda \sum_{\xi\in G} \pi_\xi \frac{\partial f}{\partial y_\xi}(y,\pi) + \gamma\lambda\sum_{\xi\in G}y_{\xi}\frac{\partial^2 f}{\partial (y_{\xi})^2}(y,\pi).
        \end{aligned}
    \ee
    Since $\pi$ is a system of independent random walks, we have
    \be{Z027a}
        \Ex[\pi_\xi(t)] = \sum_{\eta\in G} \pi_\eta(0)a_t(\eta,\xi).
    \ee
    Then we observe that the first moment $h_\xi(t)=\Ex[y_\xi(t)]$ has to fulfill the following system of differential equations, $ \forall \xi\in G, y_\xi(0)=\alpha_\xi(0)$
    \be{Z027b}
        \frac{\d}{\d t} h_\xi(t) = \sum_{\eta\in G} \big(\bar{a}(\xi,\eta)-\delta(\xi,\eta)\big)h_\eta(t) + \gamma K h_\xi(t) +  \gamma\lambda\sum_{\eta\in G} \pi_\eta(0)a_t(\eta,\xi), 
    \ee
    which is solved by
    \be{Z027c}
        h_\xi(t) = e^{\gamma Kt} \sum_{\eta\in G} \bar{a}_t(\xi,\eta) \Big( \alpha_\eta(0) + \frac{\lambda}{K}(1-e^{-\gamma Kt})\pi_\eta(0)\Big).
    \ee
    Now we calculate the right hand side of (\ref{Z026}) as
    \be{Z028}
            \Ex\Big[\sum_{\xi\in G} y_\xi(t) \bar{x}_\xi\Big]
            = e^{\gamma Kt} \sum_{\xi,\eta\in G} \bar{a}_t(\xi,\eta) \Big( \alpha_\eta(0) +  \frac{\lambda}{K}(1-e^{-\gamma Kt})\pi_\eta(0)\Big)\bar{x}_\xi.
    \ee
    The first term on the r.h.s.\ of (\ref{Z028}) can be estimated as follows
    \be{Z030}
        e^{\gamma Kt}\sum_{\xi,\eta\in G}\bar{a}_t(\xi,\eta)\alpha_\eta(0) \bar{x}_\xi \le e^{\gamma Kt}\Bigg(\sup_{\xi\in G,\atop\eta\in{\rm supp }\;\alpha(0)}\frac{a_t(\eta,\xi)}{\rho(\xi)}\Bigg) \sum_{\eta\in G} \alpha_\eta(0) \sum_{\xi\in G} \rho(\xi)\bar{x}_\xi.
    \ee
    Since $x\in\CE$ and $\alpha(0)$ has finite support, this leaves us with proving
    \be{Z031}
        \sup_{\xi\in G}\frac{a_t(\eta,\xi)}{\rho(\xi)}<\infty, \quad \forall t>0,\eta\in G.
    \ee
    To verify this note that
    \be{Z036}
        \sup_{\xi\in G}\frac{a_t(\eta,\xi)}{\rho(\xi)} \le \sup_{\xi\in G}\frac{\displaystyle e^{-t}\sum_{n=0}^\infty \frac{t^n}{n!}a^{(n)}(\eta,\xi)}{\displaystyle \sum_{n=0}^\infty C^{-n}a^{(n)}(\eta,\xi)\beta(\eta)} \le \frac{1}{\beta(\eta)}\,e^{Ct}<\infty.
    \ee
    Hence we have proved that the first term on the r.h.s.\ of (\ref{Z028}) is  finite, namely
    \be{Z034}
        e^{\gamma Kt}\sum_{\xi,\eta\in G}\bar{a}_t(\xi,\eta)\alpha_\eta(0) \bar{x}_\xi \le e^{\gamma Kt+Ct}\bigg(\max_{\eta\in{\rm supp }\;\alpha(0)}\frac{1}{\beta(\eta)}\bigg) \sum_{\eta\in G}\alpha_\eta(0) \sum_{\xi\in G}\rho(\xi)\bar{x}_\xi <\infty.
    \ee

    Now we come to the second term on the r.h.s.\ of (\ref{Z028}) which we estimate as follows (using (\ref{Z036}))
    \be{Z037}
        \begin{aligned}
            &  \frac{\lambda}{K}(e^{\gamma Kt}-1) \sum_{\xi,\eta\in G} \bar{a}_t(\xi,\eta)\pi_\eta(0)\bar{x}_\xi \\
            & =  \frac{\lambda}{K}(e^{\gamma Kt}-1) \sum_{\eta\in G} \pi_\eta(0) \sup_{\xi\in G} \frac{\bar{a}_t(\xi,\eta)}{\rho(\xi)} \sum_{\xi\in G} \rho(\xi)\bar{x}_\xi \\
            & \le  \frac{\lambda}{K}(e^{\gamma Kt}-1)e^{Ct} \left(\max_{\eta\in{\rm supp}\;\pi(0)} \frac{1}{\beta(\eta)} \right) \sum_{\eta\in G} \pi_\eta(0) \sum_{\xi\in G} \rho(\xi)\bar{x}_\xi <\infty.
        \end{aligned}
    \ee
    This completes the proof.
\end{proof}


\begin{proof}{\bf of Proposition \ref{L:dual}}\;
    Here we want to apply Theorem 4.11 in Chapter 4 of \cite{EthierKurtz1986}. First we check condition (4.42) in Chapter 4 of \cite{EthierKurtz1986} in order to see that the r.h.s.\ of (4.54) in \cite{EthierKurtz1986} is zero.
    That is, we show that
    \be{e:self3d}
        \Omega_X H\big((\alpha,\kappa),x\big) = \Omega_{(\alpha,\kappa)}H\big((\alpha,\kappa),x\big) + \beta(\alpha,\kappa)H\big((\alpha,\kappa),x\big)
    \ee
    with
    \be{beta}
    \beta(\alpha, \kappa):=  \gamma\sum_{\xi\in G}\Bigg(\sum_{m=1}^M {\kappa^{m}_{\xi}\choose 2}-\alpha_{\xi}\bar{\kappa}_{\xi}+K\bar{\kappa}_{\xi}\Bigg).
    \ee

    Since for $m=1,\dots,M,$
    \be{e:self2b}
       \frac{\partial}{\partial x_\xi^m}H\big((\alpha,\kappa),x\big)
     =
       -\alpha_{\xi} H\big((\alpha,\kappa),x\big)+\kappa^{m}_{\xi}H\big((\alpha,\kappa-e(m,\xi)),
       x\big),
    \ee
    and
    \be{e:self2c}
    \begin{aligned}
       \frac{\partial^2}{\partial(x^{m}_{\xi})^2}H\big((\alpha,\kappa),x\big)
     &=
       (\alpha_{\xi})^2
     H\big((\alpha,\kappa),x\big)-2\alpha_{\xi}\kappa^{m}_{\xi}H\big((\alpha,\kappa-e(m,\xi)),x\big)
      \\
     &\qquad\qquad+
       \kappa^{m}_{\xi}(\kappa^{m}_{\xi}-1)H\big((\alpha,\kappa-2e(m,\xi)),x\big),
    \end{aligned}
    \ee
    \be{e:self3a}
    \begin{aligned}
       &\Omega_{X}H\big((\alpha,\kappa),x\big)
      \\
     &=
       \sum_{m=1}^M \sum_{\xi,\eta\in G}
       \big({{}\bar{a}}(\xi,\eta)-\delta(\xi,\eta)\big)x^{m}_{\eta}\Big(-\alpha_{\xi} H\big((\alpha,\kappa),x\big)+\kappa^{m}_{\xi}H\big((\alpha,\kappa-e(m,\xi)),x\big)\Big)
      \\
     &\qquad+
       \gamma \sum_{m=1}^M \sum_{\xi\in G}
       x^{m}_{\xi}(K-\lambda\bar{x}_{\xi})
       \Big(-\alpha_{\xi} H\big((\alpha,\kappa),x\big)+\kappa^{m}_{\xi}H\big((\alpha,\kappa-e(m,\xi)),x\big)\Big)
      \\
     &\qquad +
       \frac{\gamma}{2} \sum_{m=1}^M \sum_{\xi\in G} x^{m}_{\xi}\bigg((\alpha_{\xi})^2 H\big((\alpha,\kappa),x\big)-2\alpha_{\xi}\kappa^{m}_{\xi}H\big((\alpha,\kappa-e(m,\xi)),x\big) \\
     &\qquad\qquad\qquad\qquad\qquad+
       \kappa^{m}_{\xi}(\kappa^{m}_{\xi}-1)H\big((\alpha,\kappa-2e(m,\xi)),x\big)\bigg).
    \end{aligned}
    \ee
    Noticing that     $x^{m}_{\xi}H\big((\alpha,\kappa),x\big)=H\big((\alpha,\kappa+e(m,\xi)),x\big)$
    yields that
    \be{e:self3b}
    \begin{aligned}
       &\Omega_{X}H\big((\alpha,\kappa),x\big)
      \\
     &=
       -\sum_{\xi\in G}\alpha_{\xi}\sum_{\eta\in G} \big({{}\bar{a}}(\xi,\eta)-\delta(\xi,\eta)\big)\bar{x}_{\eta}H\big((\alpha,\kappa),x\big)
      \\
     &\qquad
       +\sum_{m=1}^M \sum_{\xi,\eta\in G}
       \big({{}\bar{a}}(\xi,\eta)-\delta(\xi,\eta)\big)
       \kappa^{m}_{\xi}H\big((\alpha,\kappa+e(m,\eta)-e(m,\xi)),x\big)
      \\
     &\qquad
       +\gamma\sum_{\xi\in G}(K-\lambda\bar{x}_{\xi})\big(-\alpha_{\xi}
       \bar{x}_{\xi}+\bar{\kappa}_{\xi}\big)H\big((\alpha,\kappa),x\big)
      \\
     &\qquad +
       \frac{\gamma}{2}\sum_{\xi\in G}(\alpha_{\xi})^2\bar{x}_{\xi}
       H\big((\alpha,\kappa),x\big)-
       \gamma\sum_{\xi\in G}\alpha_{\xi}\bar{\kappa}_{\xi}H\big((\alpha,\kappa),x\big)
      \\
     &\qquad +
       \gamma\sum_{m=1}^M \sum_{\xi\in G} {\kappa^{m}_{\xi}\choose 2}
       H\big((\alpha,\kappa-e(m,\xi)),x\big).
    \end{aligned}
    \ee

    Moreover, since $\partial/\partial \alpha_{\xi} H=-\bar{x}_{\xi} H$ and
    $\partial^2/\partial(\alpha_{\xi})^2H=(\bar{x}_{\xi})^2 H$,
    \be{e:self3c}
    \begin{aligned}
       &\Omega_{X}H\big((\alpha,\kappa),x\big)
      \\
     &=
       \sum_{\xi\in G}\alpha_{\xi}\sum_{\eta\in G}
       \big({{}\bar{a}}(\xi,\eta)-\delta(\xi,\eta)\big)\frac{\partial}{\partial\alpha_{\eta}}
       H\big((\alpha,\kappa),x\big)
      \\
     & \qquad +\sum_{m=1}^M \sum_{\xi,\eta\in G}
       \kappa^{m}_{\xi}{{}\bar{a}}(\xi,\eta)
       \Big(H\big((\alpha,\kappa+e(m,\eta)-e(m,\xi)),x\big)-H\big((\alpha,\kappa),x\big)\Big)
      \\
     & \qquad +\gamma\sum_{\xi\in G}\bigg(K\bar{\kappa}_{\xi} H\big((\alpha,\kappa),x\big)+
       (K \alpha_{\xi}+\lambda \bar{\kappa}_{\xi})\frac{\partial}{\partial\alpha_{\xi}}
       H\big((\alpha,\kappa),x\big) \\
     & \phantom{AAAAAAAAAAAAAAAAAAAAAAAA}+\lambda\alpha_{\xi}\frac{\partial^2}{\partial(\alpha_{\xi})^2}H\big((\alpha,\kappa),x\big)\bigg)
      \\
     & \qquad -\frac{1}{2}\gamma\sum_{\xi\in G}(\alpha_{\xi})^2
       \frac{\partial}{\partial\alpha_{\xi}}H\big((\alpha,\kappa),x\big)-
       \gamma\sum_{\xi\in G}\alpha_{\xi}\bar{\kappa}_{\xi}H\big((\alpha,\kappa),x\big)
      \\
     & \qquad +\gamma\sum_{m=1}^M \sum_{\xi\in G} {\kappa^{m}_{\xi}\choose 2}
       H\big((\alpha,\kappa-e(m,\xi)),x\big)
      \\
     &=
       \Omega_{(\alpha,\kappa)}H\big((\alpha,\kappa),x\big)+
       \gamma\sum_{\xi\in G}\Bigg(\sum_{m=1}^M {\kappa^{m}_{\xi}\choose 2}-\alpha_{\xi}\bar{\kappa}_{\xi}+K\bar{\kappa}_{\xi}\Bigg)
       H\big((\alpha,\kappa),x\big),
    \end{aligned}
    \ee
    which is (\ref{e:self3d}).

    A careful inspection of the proof of Theorem 4.11 and Corollary 4.13 in Chapter 4 of \cite{EthierKurtz1986} shows that the following conditions are now sufficient to establish the duality:
    \begin{eqnarray}
        \label{s:dualcond1}
        & &  \sup_{s,t\leq T} \Ex\bigg[ \Big(\big|\beta\big(\alpha(s), \kappa(s)\big)\big|+1\Big) \Big(\big|H\big((\alpha(s),\kappa(s)),X(t)\big)\big|+1\Big)\bigg]
        <\infty,\\
          \label{s:dualcond2}
          & &\sup_{r,s,t\leq T} \Ex\bigg[ \Big(\big|\beta\big(\alpha(r), \kappa(r)\big)\big|+1\Big) \big|\Omega_{(\alpha, \kappa)} H\big((\alpha(s),\kappa(s)),X(t)\big)\big| \bigg]
          <\infty,\\
         \label{s:dualcond3}
        & &\beta\big(\alpha(r), \kappa(r)\big)\leq C\quad\forall r\le T,
    \end{eqnarray}
    where $C$ is a constant.
    Indeed, from the proof of Theorem 4.11  in Chapter 4 of \cite{EthierKurtz1986} we see that it suffices to check that for $0\leq s, s+h, t\leq T$
    \begin{eqnarray}
        \label{s:intcond1}
        && \Bigg|\int_s^{s+h} \Ex\bigg[ \bigg( \int_r^{s+h} \Omega_{(\alpha, \kappa)} H\big((\alpha(v),\kappa(v)),X(t)\big) \d v \bigg)
        \beta\big(\alpha(r), \kappa(r)\big)\\
        \nonumber
        && \phantom{AAAAAAAAAAAAAAAAAAAAa}
        \cdot \exp\bigg( \int_{0}^r \beta\big(\alpha(u), \kappa(u)\big) \d u\bigg) \bigg] \d r \Bigg| \leq C(T) h^2\\
        \label{s:intcond2}
        && \Bigg| \int_s^{s+h} \Ex\bigg[  \Omega_{(\alpha, \kappa)} H\big((\alpha(r),\kappa(r)),X(t)\big)\\
        \nonumber
        && \phantom{A}\cdot \bigg\{\exp\bigg( \int_{0}^s \beta\big(\alpha(u), \kappa(u)\big)\d u\bigg)
        - \exp\bigg( \int_{0}^r \beta\big(\alpha(u), \kappa(u)\big)\d u\bigg) \bigg\} \bigg] \d r \Bigg| \leq C(T) h^2\\
        \label{s:intcond3}
        && \Bigg|\int_0^{s} \Ex\bigg[ \Omega_{X} H\big((\alpha(t),\kappa(t)),X(r)\big)\cdot
        \exp\bigg( \int_{0}^t \beta\big(\alpha(u), \kappa(u)\big) \d u\bigg) \bigg] \d r \Bigg| \phantom{AA}\leq C(T) \\
        \label{s:intcond4}
        && \Bigg|\int_0^{s} \Ex\bigg[  \bigg\{H\big((\alpha(r),\kappa(r)),X(t)\big)\cdot  \beta\big(\alpha(r), \kappa(r)\big) + \Omega_{(\alpha, \kappa)} H\big((\alpha(r),\kappa(r)),X(t)\big) \bigg\}\\
        \nonumber
        &&\phantom{AAAAAAAAAAAAAAAAAAAA}
        \cdot \exp\bigg( \int_{0}^r \beta\big(\alpha(u), \kappa(u)\big) \d u\bigg) \bigg] \d r \Bigg| \leq C(T)
    \end{eqnarray}
    Here, (\ref{s:intcond3}) and (\ref{s:intcond4}) follow immediately from (\ref{e:self3d}) to (\ref{s:dualcond3}). Also, due to (\ref{s:dualcond3}) the expression on the l.h.s.\ in (\ref{s:intcond1}) is bounded by
    \be{s:intcond1check}
    \begin{aligned}
        & e^{C(s+h)} \int_s^{s+h}\int_r^{s+h}  \Ex\bigg[ \Big|\Omega_{(\alpha, \kappa)} H\big((\alpha(v),\kappa(v)),X(t)\big)\Big| \cdot
        \Big| \beta\big(\alpha(r), \kappa(r) \big)\Big| \bigg] \d v\, \d r \\
        &\leq e^{C(s+h)} h^2 \sup_{s\le r\le v\le s+h} \Ex\bigg[\Big| \beta\big(\alpha(r), \kappa(r) \big) \Big| \cdot \Big|\Omega_{(\alpha, \kappa)} H\big((\alpha(v),\kappa(v)),X(t)\big)\Big| \bigg] \\
        &\leq C(T) h^2,
    \end{aligned}
    \ee
    where we have used  (\ref{s:dualcond2}). The expression on the l.h.s.\ in (\ref{s:intcond2}) can be bounded as follows
    \be{s:intcond2check}
    \begin{aligned}
        & \Bigg| \int_s^{s+h} \Ex\Bigg[  \Omega_{(\alpha, \kappa)} \big((\alpha(r),\kappa(r)),X(t)\big)\\
        & \phantom{AA}\cdot \exp\bigg( \int_{0}^s \beta\big(\alpha(u), \kappa(u)\big)\d u\bigg)\cdot\bigg\{1
        - \exp\bigg( \int_{s}^r \beta\big(\alpha(u), \kappa(u)\big)\d u\bigg) \bigg\} \Bigg] \d r \Bigg| \\
        & \le e^{Cs}\int_s^{s+h} \Ex\Bigg[  \Big|\Omega_{(\alpha, \kappa)} H\big((\alpha(r),\kappa(r)),X(t)\big)\Big|\\
        &\phantom{AAAAAA}\cdot  \Big|\int_{s}^r \beta\big(\alpha(u), \kappa(u)\big) \d u\Big| \cdot
        \exp\bigg( 0\vee \sup_{s\le v\le s+h}  \int_s^v \beta\big(\alpha(u), \kappa(u)\big) \d u\bigg) \Bigg]\d r \\
        &\leq  e^{Cs}\dot\exp\bigg( 0\vee \sup_{s\le v\le s+h}  \int_s^v \beta\big(\alpha(u), \kappa(u)\big) \d u\bigg) \\
        &\phantom{AAAAAA} \cdot \int_s^{s+h} \int_{s}^r \Ex\bigg[  \Big|\Omega_{(\alpha, \kappa)} H\big((\alpha(r),\kappa(r)),X(t)\big)\Big| \cdot  \Big| \beta\big(\alpha(u), \kappa(u)\big) \Big| \bigg] \d u \,\d r \\
        &\leq C(T) h^2,
    \end{aligned}
    \ee
    where we first applied (\ref{s:dualcond3}) and we used the fact $|1-e^x|\le |x|e^{0\vee x}$ and then we applied (\ref{s:dualcond2}).

    The remainder of the proof is therefore concerned with showing the three integrability conditions
    (\ref{s:dualcond1}) to (\ref{s:dualcond3}). We first note that condition (\ref{s:dualcond3}) is naturally fulfilled since we start with $n= \sum_{\xi \in G} \bar{\kappa}_{\xi}(0) < \infty,$ and $\sum_{\xi \in G} \bar{\kappa}_{\xi}(t)\leq n$ for all $t \geq 0.$
    In order to show (\ref{s:dualcond1}) and (\ref{s:dualcond2}) we define
    \be{Akappa}
    A_{\kappa}(t):=\{ \xi \in G: \bar{\kappa}_{\xi}(t) >0\} \quad \quad \text{ and }  \quad \quad \bar{A}_{\kappa}(T)=\bigcup_{0\leq t\leq T}
    A_{\kappa}(t)
    \ee
    as the set of sites occupied by $\kappa$ particles at time $t$ and up to time $T$ respectively.
    We write $|\bar{A}_{\kappa}(T)|$ for the cardinality of the latter set. We let ${\cal F}^{\kappa}_t=\sigma( (\kappa(s))_{s\leq t})$ be the $\sigma$-algebra generated by  the process $\kappa$ up to time $t.$
    The following estimates will be used repeatedly. The constant $C$ may change from line to line. First note that
    \begin{equation}
        \label{s:betabound0}
         \big|\beta\big(\alpha(r), \kappa(r)\big)\big|
         \leq \gamma \bigg( n^2 + K n + \sum_{\xi \in G} \alpha_{\xi}(r) \bar{\kappa}_{\xi}(r)\bigg)
        \leq C \bigg(1+  \sum_{\xi \in G} \alpha_{\xi}(r)\bigg).
    \end{equation}
    We also have  for any fixed $(m_\xi)_{\xi \in G} \in \N^G$ with $m=\sum_{\xi \in G} m_\xi$
    and $A= \{ \xi \in G: m_{\xi}\geq 1\}$ finite that
    \be{s:prodsum}
    \begin{aligned}
         \prod_{\xi \in G} \bar{x}_\xi(t)^{m_{\xi}} x(t)^{\kappa(s)}
         &\leq \prod_{\xi \in G} \bar{x}_{\xi}(t)^{m_\xi + \bar{\kappa}_{\xi}(s)}
         \leq \prod_{\xi \in G}\Big(\max_{\zeta \in A_{\kappa}(s)\cup A}         \bar{x}_{\zeta}(t)\Big)^{m_\xi+ \bar{\kappa}_{\xi}(s)}   \\
         &\leq \Big( 1 \vee\max_{\xi \in A_{\kappa}(s)\cup A} \bar{x}_{\xi}(t)\Big)^{m+n} \leq  1+\sum_{\xi \in A_{\kappa}(s)\cup A}  \bar{x}_{\xi}(t)^{m+n}
    \end{aligned}
    \ee
    where $a\vee b$ denotes the maximum of $a$ and $b$.
    Therefore,
    \be{s:x^kappabound}
    \begin{aligned}
        & \Ex\bigg[  \sup_{s,t \leq T} \prod_{\xi \in G} \bar{x}_\xi(t)^{m_{\xi}} x(t)^{\kappa(s)} \Big| {\cal F}^{\kappa}_T \bigg]  \leq \Ex\bigg[  \sup_{s,t \leq T} \Big(1+\sum_{\xi \in \bar{A}_{\kappa}(T) \cup A} \bar{x}_{\xi}(t)^{m+n }\Big)\Big| {\cal F}^{\kappa}_T\bigg]\\
        &\leq 1+ \Big(|\bar{A}_{\kappa}(T)|+ |A|\Big)
        \sup_{\xi \in G} \Ex\bigg[\sup_{t \leq T}  \bar{x}_{\xi}(t)^{m+n}\bigg]\leq C\Big( 1+|\bar{A}_{\kappa}(T)|+ |A|\Big)
    \end{aligned}
    \ee
    by Proposition \ref{PMom-exch} provided that $\bar{x}(0)$ is bounded above by a translation invariant distribution with bounded $(m+n)$-th moment at each site. We will use (\ref{s:x^kappabound})
    with $m\leq 2.$

    To check (\ref{s:dualcond1})
    we condition on ${\cal F}^{\kappa}_T$ and use the conditional independence of $\alpha$ and
    $X$ as well as (\ref{s:betabound0}) and (\ref{s:x^kappabound})  to arrive at
    \be{s:betaH}
    \begin{aligned}
        & \sup_{s,t\leq T}\Ex\Big[  \big|\beta\big(\alpha(s), \kappa(s)\big)\big|\cdot\big|H\big((\alpha(s),\kappa(s)),X(t)\big)\big|\Big]\\
        &\leq C \sup_{s,t\leq T}
        \Ex\bigg[  \Ex\Big[1+\sum_{\xi \in G} \alpha_{\xi}(s) \Big| {\cal F}^{\kappa}_T \Big]
        \cdot  \Ex\Big[ x(t)^{\kappa(s)} \Big| {\cal F}^{\kappa}_T  \Big]  \bigg]\\
        &\leq C \,\,
        \Ex\bigg[ \sup_{s\leq T}  \Ex\Big[1+\sum_{\xi \in G} \alpha_{\xi}(s) \Big| {\cal F}^{\kappa}_T \Big]
        \cdot  \Big(1+\big|\bar{A}_{\kappa}(T)\big|\Big)  \bigg]
    \end{aligned}
    \ee
    By  Proposition \ref{monotonicity}
    we can couple $\sum_{\xi \in G} \alpha_{\xi}$ to a
     supercritical nonspatial branching process $\bar{\alpha}$ with immigration solving
     \begin{equation}
     \label{s:baralpha}
     \d \bar{\alpha}(t)
     = \gamma K \bar{\alpha}(t)\,\d t + \sqrt{\gamma \lambda \bar{\alpha}(t)} \,\d w(t)+ \gamma \lambda n \,\d t,
     \end{equation}
     such that for $\sum_{\xi \in G} \alpha_{\xi}(0)=\bar{\alpha}(0)$ we have $\sum_{\xi \in G} \alpha_{\xi}(t)\leq \bar{\alpha}(t)$ for $t \geq 0.$ We note that $ \bar{\alpha}$ is independent of $\kappa.$ Since for any $m\in \N,$ $\Ex[\bar{\alpha}(0)^m]<\infty$ implies $\sup_{t \leq T} \Ex[\bar{\alpha}(t)^m]<\infty$
     we have
    \begin{equation}
        \label{s:alphasum}
        \sup_{r\leq T}  \Ex\bigg[\Big(\sum_{\zeta \in G} \alpha_{\zeta}(r)\Big)^m   \Big| {\cal F}^{\kappa}_T  \bigg]
        \leq \sup_{r\leq T}  \Ex\big[ \bar{\alpha}(r)^m  \big]< \infty
    \end{equation}
    provided that $\Ex\big[(\sum_{\xi \in G} \alpha_{\xi}(0) )^m\big]< \infty.$
    Hence, we can bound (\ref{s:betaH}) by $C\big(1+ \Ex\big[ |\bar{A}_{\kappa}(T)|\big]\big).$ But since each $\kappa$ particle performs an independent random walk (until coalescence) at rate one
    the number of sites in the set $\bar{A}_{\kappa}(T)$, $|\bar{A}_{\kappa}(T)|,$ can be bounded by a Poisson random variable with parameter $nT$ and so we arrive at (\ref{s:dualcond1}).

    We now turn to showing (\ref{s:dualcond2}) with similar means. Recall the form of the generator in (\ref{e:thedual2}), which consists of six terms. The first term we bound as follows (using $\bar{\kappa}_\xi\leq n$ and (\ref{s:prodsum}),
    \be{s:B_1bound}
    \begin{aligned}
        &|B_1(s,t)| \\
        &:= \bigg|\sum_{m=1}^M \sum_{\xi,\eta\in G} \kappa^{m}_{\xi}(s) a(\xi,\eta) \exp\bigg(\!\!-\!\sum_{\zeta \in G}\alpha_\zeta(s) \bar{x}_{\zeta}(t) \bigg) \Big(x(t)^{\kappa(s)+e_{\eta}^m-e_{\xi}^m }- x(t)^{\kappa(s)}\Big)\bigg|\\
        &\leq \sum_{m=1}^M\sum_{\xi,\eta\in G} \kappa_{\xi}^m(s) a(\xi,\eta) \Big|x(t)^{\kappa(s)+e_{\eta}^m-e_{\xi}^m }- x(t)^{\kappa(s)}\Big|\\
        &\leq 2n  \sum_{\xi \in \bar{A}_{\kappa}(T)} \sum_{\eta\in G} a(\xi,\eta)
        \bigg(1+\sum_{\zeta \in \bar{A}_{\kappa}(T) \cup\{\eta\}}
        \bar{x}_{\zeta}(t)^{n }\bigg)
    \end{aligned}
    \ee
    which depends on $(\alpha, \kappa)$ only through $\bar{A}_{\kappa}(T).$
    Hence, by (\ref{s:betabound0}) and (\ref{s:B_1bound}) and conditioning on ${\cal F}^{\kappa}_T$,
    \be{s:betaB_1bound}
    \begin{aligned}
        & \sup_{r,s,t\leq T}\Ex\bigg[\Big(\big|\beta\big(\alpha(r), \kappa(r)\big)\big|+1\Big)\big|B_1(s,t)\big|\bigg]\\
        &\leq C \sup_{r,t\leq T}\Ex\bigg[ \Ex\Big[1\!+\!\sum_{\zeta \in G} \alpha_{\zeta}(r) \Big| {\cal F}^{\kappa}_T  \Big] \! \cdot    \Ex\Big[\!\sum_{\xi \in \bar{A}_{\kappa}(T)} \sum_{\eta\in G} a(\xi,\eta)
        \big(1+\!\!\!\sum_{\zeta \in \bar{A}_{\kappa}(T) \cup\{\eta\}} \!\!\! \bar{x}_{\zeta}(t)^{n }\big) \Big| {\cal F}^{\kappa}_T  \Big]  \bigg]\\
        &\leq C\,\sup_{t\leq T}\, \Ex\bigg[ \sum_{\xi \in \bar{A}_{\kappa}(T)} \sum_{\eta\in G} a(\xi,\eta) \Ex\Big[  1+\sum_{\zeta \in \bar{A}_{\kappa}(T) \cup\{\eta\}} \bar{x}_{\zeta}(t)^{n }  \Big| {\cal F}^{\kappa}_T \Big]\bigg]\\
        &\leq C \, \Ex\bigg[ \sum_{\xi \in \bar{A}_{\kappa}(T)} \sum_{\eta\in G} a(\xi,\eta) \Big(1+|\bar{A}_{\kappa}(T)|\Big) \bigg] \leq  C\Ex\Big[ \big(1+|\bar{A}_{\kappa}(T)|\big)^2\Big]< \infty,
    \end{aligned}
    \ee
    where we have used (\ref{s:alphasum}) and (\ref{s:x^kappabound}) in the second and third inequality respectively. Finally we used again that $|\bar{A}_{\kappa}(T)|$ can be bounded by a Poisson random variable with parameter $nT$.

    For the second term we use (\ref{s:prodsum}) to obtain
    \be{s:B_2bound}
    \begin{aligned}
        |B_2(s,t)|
        &:= \bigg|\gamma \sum_{m=1}^M\sum_{\xi\in G}{\kappa^{m}_{\xi}(s)\choose 2}
        \exp\bigg(-\sum_{\zeta \in G}\alpha_\zeta(s) \bar{x}_{\zeta}(t) \bigg)\Big(x(t)^{\kappa(s)-e_{\xi}^m}- x(t)^{\kappa(s)}\Big)\bigg| \\
        &\leq 2\gamma n^2 \bigg(1+ \sum_{\xi \in \bar{A}_{\kappa}(T)} \bar{x}_{\xi}(t)^{n}\bigg).
    \end{aligned}
    \ee
    Hence by (\ref{s:betabound0}) and (\ref{s:B_2bound}) we obtain
    \be{s:betaB_2bound}
    \begin{aligned}
        & \sup_{r,s,t\leq T} \Ex\Big[ \Big(\big|\beta\big(\alpha(r), \kappa(r)\big)\big|+1\Big)\big|B_2(s,t)\big| \Big] \\
        &\leq C\sup_{r,s,t\leq T} \Ex\bigg[ \Ex\Big[ 1+ \sum_{\zeta \in G} \alpha_{\zeta}(r) \Big| {\cal F}^{\kappa}_T \Big]
         \cdot \Ex\Big[ 1+ \sum_{\xi \in \bar{A}_{\kappa}(T)} \bar{x}_{\xi}(t)^{n}\Big| {\cal F}^{\kappa}_T \Big] \bigg]\\
        &\leq C  \sup_{r\leq T} \Ex\big[ 1+ \bar{\alpha}(r)\big] \cdot \Ex\big[1+ \bar{A}_{\kappa}(T)\big]<\infty,
    \end{aligned}
    \ee
    where we have used (\ref{s:x^kappabound}) and (\ref{s:alphasum}).

    For the third term we bound
    \be{s:B_3bound}
    \begin{aligned}
        |B_3(s,t)|
        &:= \bigg| \sum_{\xi,\eta\in G}\bar{a}(\xi,\eta)\big(\alpha_\eta(s)-\alpha_{\xi}(s)\big) \exp\bigg(\!\!-\!\sum_{\zeta \in G}\alpha_\zeta(s) \bar{x}_{\zeta}(t) \bigg) \big(-\bar{x}_{\xi}(t)\big) x(t)^{\kappa(s)}\bigg|\\
        &\leq  \sum_{\xi, \eta\in G}\bar{a}(\xi,\eta) \big(\alpha_{\eta}(s)+ \alpha_{\xi}(s)\big) \bar{x}_{\xi}(t)
         x(t)^{\kappa(s)}
    \end{aligned}
    \ee
    By conditioning on ${\cal F}^{\kappa}_T$ and using (\ref{s:betabound0})
    as well as (\ref{s:B_3bound}) we obtain
    \be{s:betaB_3bound}
    \begin{aligned}
        & \sup_{r,s,t\leq T} \Ex\Big[\Big(\big|\beta\big(\alpha(r), \kappa(r)\big)\big|+1\Big)\big|B_3(s,t)\big|\Big]\\
        &\leq C \sup_{r,s,t\leq T} \Ex\bigg[ \Big(1+ \sum_{\zeta \in G} \alpha_{\zeta}(r) \Big)  \cdot \bigg(  \sum_{\xi, \eta\in G}\bar{a}(\xi,\eta) \big(\alpha_{\eta}(s)+ \alpha_{\xi}(s)\big) \bar{x}_{\xi}(t) x(t)^{\kappa(s)} \bigg)\bigg]\\
        &= C  \sup_{r,s,t\leq T}\Ex\Bigg[  \sum_{\xi, \eta\in G} \bar{a}(\xi,\eta)  \cdot \Ex\Big[\Big(1+ \sum_{\zeta \in G} \alpha_{\zeta}(r)\Big)\big(\alpha_{\eta}(s)+ \alpha_{\xi}(s)\big)\Big| {\cal F}^{\kappa}_T \Big] \\
        & \phantom{AAAAAAAAAAAAAAAAA}\cdot \Ex\Big[ \bar{x}_{\xi}(t) x(t)^{\kappa(s)} \Big| {\cal F}^{\kappa}_T \Big] \Bigg]\\
        &\leq C\, \Ex\bigg[ \sup_{r,s\leq T}\sum_{\xi, \eta\in G} \bar{a}(\xi,\eta)  \cdot \Ex\Big[\Big(1+ \sum_{\zeta \in G} \alpha_{\zeta}(r)\Big)\big(\alpha_{\eta}(s)+ \alpha_{\xi}(s)\big) \Big| {\cal F}^{\kappa}_T \Big]  \cdot \big(1+\bar{A}_{\kappa}(T)\big) \bigg],
    \end{aligned}
    \ee
    where we have used  (\ref{s:x^kappabound}) in the last inequality. But we may bound
    \be{s:B_3alphabound}
    \begin{aligned}
        & \sup_{r,s\leq T} \sum_{\xi, \eta\in G} \bar{a}(\xi,\eta)  \Ex\bigg[\Big ( 1+\sum_{\zeta \in G} \alpha_{\zeta}(r)\Big) \cdot \big(\alpha_{\eta}(s)+ \alpha_{\xi}(s)\big) \Big| {\cal F}^{\kappa}_T \bigg] \\
        &\leq 2 \sup_{r,s\leq T}  \Ex\bigg[ \sum_{\xi \in G} \alpha_{\xi}(s) +  \sum_{\zeta, \xi \in G} \alpha_{\zeta}(r) \alpha_{\xi}(s) \Big| {\cal F}^{\kappa}_T \bigg]\\
        &\leq 4 \sup_{s\leq T} \Ex\bigg[ \sum_{\xi \in G} \alpha_{\xi}(s) +  \Big(\sum_{\xi \in G} \alpha_{\xi}(s) \Big)^2 \Big| {\cal F}^{\kappa}_T \bigg]\leq 4 \sup_{s\leq T} \Ex\Big[ \bar{\alpha}(s) + \bar{\alpha}(s)^2\Big]< \infty
    \end{aligned}
    \ee
    by (\ref{s:alphasum}) provided that   $\Ex[(\sum_{\xi \in G} \alpha_{\xi}(0))^2]<\infty.$
    Hence, the expression in (\ref{s:betaB_3bound}) is finite as well.

    For the fourth term we use that $x \exp(-x) \leq C$ for $x\geq 0$ and obtain
    \be{s:B_4bound}
    \begin{aligned}
        |B_4(s,t)|
        &:= \bigg| \gamma \sum_{\xi \in G} \alpha_{\xi}(s)\Big( K -\frac{1}{2} \alpha_{\xi}(s)\Big) \exp\bigg(-\sum_{\zeta\in G}\alpha_\zeta(s) \bar{x}_{\zeta}(t) \bigg) \big(-\bar{x}_{\xi}(t)\big) x(t)^{\kappa(s)}\bigg| \\
        &\le \gamma K \sum_{\xi \in G} \alpha_{\xi}(s)  \bar{x}_{\xi}(t)  \exp\bigg(-\sum_{\zeta\in G}\alpha_\zeta(s) \bar{x}_{\zeta}(t) \bigg)x(t)^{\kappa(s)}\\
        & \phantom{AAA}+  \frac{1}{2} \gamma \sum_{\xi \in G} \alpha_{\xi}(s)^2 \bar{x}_{\xi}(t) x(t)^{\kappa(s)}\\
        &\leq C \bigg(x(t)^{\kappa(s)}+ \sum_{\xi \in G} \alpha_{\xi}(s)^2 \bar{x}_{\xi}(t) x(t)^{\kappa(s)}\bigg)
    \end{aligned}
    \ee
    Therefore, by (\ref{s:betabound0}) and (\ref{s:B_4bound})
    \be{s:betaB_4bound}
    \begin{aligned}
        & \sup_{r,s,t\leq T} \Ex\Big[ \Big(\big|\beta\big(\alpha(r), \kappa(r)\big)\big|+1\Big)\big|B_4(s,t)\big|\Big]\\
        &\leq C \sup_{r,s,t\leq T} \Ex\bigg[\Big(1+ \sum_{\zeta \in G} \alpha_{\zeta}(r) \Big)\cdot \Big(x(t)^{\kappa(s) }+   \sum_{\xi \in G} \alpha_{\xi}(s)^2 \bar{x}_{\xi}(t) x(t)^{\kappa(s)}\Big)\bigg]\\
        &\leq C\Bigg(1+ \sup_{r,s,t\leq T} \sum_{\xi \in G} \Ex\bigg[ \Ex\Big[ \Big(1+ \sum_{\zeta \in G} \alpha_{\zeta}(r) \Big) \alpha_{\xi}(s)^2  \Big| {\cal F}^{\kappa}_T \Big] \cdot  \Ex\Big[ \bar{x}_{\xi}(t) x(t)^{\kappa(s)} \Big| {\cal F}^{\kappa}_T \Big]\bigg]\Bigg) \\
        &\leq C\Bigg(1+ \sup_{r,s\leq T}  \Ex\bigg[ \Ex\Big[  \sum_{\xi \in G}\Big(1+ \sum_{\zeta \in G} \alpha_{\zeta}(r) \Big) \alpha_{\xi}(s)^2  \Big| {\cal F}^{\kappa}_T \Big] \cdot \Big( 1+\bar{A}_{\kappa}(T)\Big)\bigg] \Bigg)\\
        &\leq C\Big(1+ 2 \sup_{s\leq T}  \Ex\Big[  \bar{\alpha}(s)^2 + \bar{\alpha}(s)^3 \Big] \cdot \Ex\big[ 1+\bar{A}_{\kappa}(T)\big] \Big)
    \end{aligned}
    \ee
    where we have used (\ref{s:x^kappabound}) and (\ref{s:alphasum}) in the second inequality, then (\ref{s:x^kappabound}) again in the third inequality. The last calculation is similar to that in (\ref{s:B_3alphabound}).
    Hence, the term in (\ref{s:betaB_4bound}) is finite provided that
    $ \Ex[ (\sum_{\xi \in G}\alpha_{\xi}(0))^3]< \infty.$

    For the fifth term we need to bound
    \be{s:B_5bound}
    \begin{aligned}
        |B_5(s,t)|
        &:= \Bigg| \gamma \lambda \sum_{\xi \in G} \alpha_{\xi}(s) \exp\bigg(-\sum_{\zeta \in G}\alpha_\zeta(s) \bar{x}_{\zeta}(t) \bigg) \big(\bar{x}_{\xi}(t)\big)^2 x(t)^{\kappa(s)}\Bigg|\\
        &\leq  C \sum_{\xi \in G} \alpha_{\xi}(s)\big(\bar{x}_{\xi}(t)\big)^2 x(t)^{\kappa(s)}
    \end{aligned}
    \ee
    Therefore, again  by (\ref{s:betabound0}),  (\ref{s:x^kappabound}), and (\ref{s:alphasum})
    \be{s:betaB_5bound}
    \begin{aligned}
        & \sup_{r,s,t\leq T} \Ex\Big[ \Big(\big|\beta\big(\alpha(r), \kappa(r)\big)\big|+1\Big) \big|B_5(s,t)\big| \Big]\\
        &\leq C \sup_{r,s,t\leq T} \Ex\bigg[ \Big(1+ \sum_{\zeta \in G} \alpha_{\zeta}(r) \Big)\cdot \Big(  \sum_{\xi \in G} \alpha_{\xi}(s)(\bar{x}_{\xi}(t))^2 x(t)^{\kappa(s)} \Big) \bigg]\\
        &= C \sup_{r,s,t\leq T} \Ex\Bigg[ \sum_{\xi \in G} \Ex\bigg[ \alpha_\xi(s) \bigg(1+ \sum_{\zeta \in G} \alpha_{\zeta}(r)\bigg)\bigg| {\cal F}^{\kappa}_T \bigg]  \cdot \Ex\bigg[ \bar{x}_{\xi}(t)^2 x(t)^{\kappa(s)} \Big| {\cal F}^{\kappa}_T \bigg] \Bigg]\\
        &\leq C\,\Ex\Bigg[ \sup_{r,s\leq T} \Ex\bigg[ \sum_{\xi \in G}\alpha_\xi(s) \bigg(1+ \sum_{\zeta \in G} \alpha_{\zeta}(r)\bigg) \bigg| {\cal F}^{\kappa}_T \bigg] \cdot \bigg(1+\bar{A}_{\kappa}(T)\bigg) \Bigg]\\
        &\leq C \sup_{s\leq T} \Ex\Big[ \bar{\alpha}(s) + \bar{\alpha}(s)^2\Big] \cdot  \Ex\Big[1+\bar{A}_{\kappa}(T)\Big]<\infty.
    \end{aligned}
    \ee
    For the sixth term we need to bound
    \be{s:B_6bound}
    \begin{aligned}
        |B_6(s,t)|
        &:= \Bigg|  \gamma \lambda \sum_{\xi \in G} \bar{\kappa}_{\xi}(s) \exp\bigg(-\sum_{\zeta \in G}\alpha_\zeta(s) \bar{x}_{\zeta}(t) \bigg) \big(-\bar{x}_{\xi}(t)\big) x(t)^{\kappa(s)}\Bigg|\\
        &\leq  C \sum_{\xi \in G} \bar{\kappa}_{\xi}(s) \bar{x}_{\xi}(t) x(t)^{\kappa(s)}
    \end{aligned}
    \ee
    Hence,  by  (\ref{s:betabound0}), (\ref{s:x^kappabound}) and (\ref{s:alphasum})
    we obtain  that
    \be{s:betaB_6bound}
    \begin{aligned}
        & \sup_{r,s,t\leq T} \Ex\Big[ \Big(\big|\beta\big(\alpha(r), \kappa(r)\big)\big|+1\Big)\big|B_6(s,t)\big|\Big]\\
        &\leq C \sup_{r,s,t\leq T} \Ex\bigg[\bigg(1+ \sum_{\zeta \in G} \alpha_{\zeta}(r) \bigg) \cdot \bigg(\sum_{\xi \in G} \bar{\kappa}_{\xi}(s)  \bar{x}_{\xi}(t) x(t)^{\kappa(s)} \bigg)\bigg]\\
        &= C \sup_{r,s,t\leq T} \Ex\Bigg[ \sum_{\xi \in G} \bar{\kappa}_{\xi}(s) \Ex\bigg[ 1+ \sum_{\zeta \in G} \alpha_{\zeta}(r) \bigg | {\cal F}^{\kappa}_T \bigg] \cdot \Ex\bigg[ \bar{x}_{\xi}(t) x(t)^{\kappa(s)} \bigg| {\cal F}^{\kappa}_T \bigg] \Bigg]\\
        &\le C \sup_{s\leq T} \Ex\bigg[ \sum_{\xi \in G} \bar{\kappa}_{\xi}(s) \sup_{r\leq T}  \Ex\Big[ 1+ \bar{\alpha}(r) \Big]  \cdot \Big(1+\bar{A}_{\kappa}(T)\Big) \bigg]\\
        &\leq Cn \sup_{r\leq T} \Ex\Big[  1+ \bar{\alpha}(r)\Big] \cdot \Ex\Big[1+\bar{A}_{\kappa}(T)\Big]< \infty.
    \end{aligned}
    \ee
    This completes the proof of (\ref{s:dualcond2}) and so establishes the duality.
\end{proof}

\bigskip

\noindent
{\em Acknowledgment. }We acknowledge many helpful discussions with Peter Pfaffelhuber, and
support by DFG within SFB/TR 12.

\bibliographystyle{alpha}
\bibliography{gpswz}

\begin{thebibliography}{ABBP02}

\bibitem[ABBP02]{AtBarBassPer02}
S.~Athreya, M.~Barlow, R.~Bass, and E.~Perkins.
\newblock Degenerate stochastic differential equations and super-{M}arkov
  chains.
\newblock {\em Probability Theory and Related Fields}, 123:484--520, 2002.

\bibitem[ABP05]{AtBassPer05}
S.~Athreya, R.~Bass, and E.~Perkins.
\newblock {H}\"older norm estimates for elliptic operators on finite and
  infinite dimensional spaces.
\newblock {\em Trans. Amer. Math. Soc.}, 357:5001--5029, 2005.

\bibitem[BEM07]{BlathEtheridgeMeredith2007}
J.~Blath, A.~Etheridge, and M.~Meredith.
\newblock Coexistence in locally regulated competing populations and survival
  of branching annihilating random walk.
\newblock {\em Annals Appl. Probab.}, 17(5/6):1474--1507, 2007.

\bibitem[BP97]{BolkerPacala1997}
B.M. Bolker and S.W. Pacala.
\newblock Using moment equations to understand stochastically driven spatial
  pattern formation in ecological systems.
\newblock {\em Theoretical Population Biology}, 52:179--197, 1997.

\bibitem[BP99]{BolkerPacala1999}
B.M. Bolker and S.W. Pacala.
\newblock Spatial moment equations for plant competition: {U}nderstanding
  spatial strategies and the advantages of short dispersal.
\newblock {\em American naturalist}, 153:575--602, 1999.

\bibitem[BZ07]{BirknerZaehle2007}
M.~Birkner and I.~Z{\"a}hle.
\newblock A functional {CLT} for the occupation time of state-dependent
  branching random walk.
\newblock {\em Ann. Probab.}, 35(6):2063--2090, 2007.

\bibitem[CG94]{CoxGreven1994a}
J.~T. Cox and A.~Greven.
\newblock Ergodic theorems for infinite systems of locally interacting
  diffusions.
\newblock {\em Ann. Probab.}, 22(2):833--853, 1994.

\bibitem[CP05]{CoxPerkins2005}
J.~T. Cox and E.A. Perkins.
\newblock Rescaled {L}otka-{V}olterra models converge to super-{B}rownian
  motion.
\newblock {\em Annals Probab.}, 33(3):904--947, 2005.

\bibitem[EK86]{EthierKurtz1986}
S.~N. Ethier and T.~G. Kurtz.
\newblock {\em Markov processes, characterization and convergence}.
\newblock John Wiley \& Sons Inc., New York, 1986.

\bibitem[Eth04]{Etheridge2004}
A.M. Etheridge.
\newblock Survival and {E}xtinction in a locally regulated population.
\newblock {\em Ann. Appl. Probab.}, 14(1):188--214, 2004.

\bibitem[HW07]{HW07}
M.~Hutzenthaler and A.~Wakolbinger.
\newblock Ergodic behaviour of locally regulated branching populations.
\newblock {\em Ann. Appl. Probab.}, 17(2), 2007.

\bibitem[JK02]{JagersKlebaner2000}
P.~Jagers and F.~C. Klebaner.
\newblock Population-size-dependent and age-dependent.
\newblock {\em Stoch. Proc. Appl.}, 87:235--254, 2002.

\bibitem[LS81]{LiggettSpitzer1981}
T.M. Liggett and F.~Spitzer.
\newblock Ergodic theorems for coupled random walks and other systems with
  locally interacting components.
\newblock {\em Z. Wahrsch. verw. Gebiete}, 56:443--468, 1981.

\bibitem[MT94]{MuellerTribe1994}
C.~M\"uller and R.~Tribe.
\newblock A phase transition for a stochastic pde related to the contact
  process.
\newblock {\em Probab. Theory Rel. Fields}, 100:131--156, 1994.

\bibitem[NP99]{NeuhauserPacala1999}
C.~Neuhauser and S.W. Pacala.
\newblock An explicitely spatial version of the {L}otka-{V}olterra model with
  interspecific competition.
\newblock {\em Annals of applied probability}, 9(4):1226--1259, 1999.

\bibitem[SS80]{ShigaShimizu1980}
T.~Shiga and A.~Shimizu.
\newblock Infinite dimensional stochastic differential equations and their
  applications.
\newblock {\em J. Math. Kyoto Univ.}, 20-3:395--416, 1980.

\bibitem[Stu03]{Sturm03}
A.~Sturm.
\newblock On convergence of population processes in random environments to the
  stochastic heat equation with colored noise.
\newblock {\em Electronic Journal of Probability}, 8(6):1--39, 2003.

\bibitem[YW71]{YamadaWatanabe1971}
T.~Yamada and S.~Watanabe.
\newblock On the uniqueness of solutions of stochastic differential equations.
\newblock {\em J. Math. Kyoto Univ.}, 11:155--167, 1971.

\end{thebibliography}

\vspace{1cm}
{\setlength{\parindent}{0mm}
\footnotesize\sc
\begin{minipage}[t]{7.5cm}
Andreas Greven\\ Mathematisches Institut\\ Universit\"at Erlangen--N\"urnberg\\ Cauerstr. 11
\\
91058 Erlangen \\ Germany \\ E-Mail: {\rm greven@mi.uni-erlangen.de}
\end{minipage}
\hfill
\begin{minipage}[t]{7.5cm}
%
%
%
Anja Sturm \\
Institut für Mathematische Stochastik\\
Georg-August-Universit\"at G\"ottingen\\
Goldschmidtstra{\ss}e 7\\
37077 G\"ottingen\\ Germany\\
E-Mail: {\rm asturm@math.uni-goettingen.de }
\end{minipage}
\hfill

\vspace{1cm}
\begin{minipage}[t]{7.5cm}
Anita Winter\\ Fakult\"at f\"ur Mathematik \\
Universit\"at Duisburg-Essen \\
Thea-Leymann-Stra{\ss}e 9 \\
45127 Essen
 \\
Germany\\
E-Mail: {\rm anita.winter@uni-due.de}
\end{minipage}

\vspace{1cm}
}

\end{document}